\documentclass[a4paper, 11pt]{article}

\usepackage[utf8]{inputenc}
\usepackage{amsmath}
\usepackage{amsthm}
\usepackage{amsfonts}
\usepackage{amssymb}
\usepackage{amstext}
\usepackage{dsfont}
\usepackage{graphicx}
\usepackage{color}
\usepackage[colorlinks]{hyperref}
\usepackage{epigraph}
\usepackage[left=3.5cm,top=2cm,bottom=3.5cm,right=3cm,nohead,nofoot]{geometry}
\usepackage{tikz}
\usepackage{tkz-graph}
\usepackage{tkz-berge}
\usetikzlibrary{arrows,shapes,matrix}


\title{Non-Archimedean Whittaker functions as characters: a probabilistic approach to the Shintani-Casselman-Shalika formula}
\author{Reda \textsc{Chhaibi}        \footnote{\texttt{reda.chhaibi@math.uzh.ch}} } %\footnotemark
\date{}

\allowdisplaybreaks[4]

\DeclareMathOperator{\Card}{ \textrm{Card} }
\DeclareMathOperator{\ch}{ \textrm{ch} }
\DeclareMathOperator{\eqlaw}{\stackrel{\Lc}{=}}
\DeclareMathOperator{\Hom}{ \textrm{Hom} }
\DeclareMathOperator{\Id}{Id}
\DeclareMathOperator{\height}{ht}
\DeclareMathOperator{\Inv}{Inv}
\DeclareMathOperator{\Ker}{Ker}

\DeclareMathOperator{\val}{val}

\def\half{\frac{1}{2}}

\def\N{{\mathbb N}}
\def\Z{{\mathbb Z}}
\def\Q{{\mathbb Q}}
\def\R{{\mathbb R}}
\def\C{{\mathbb C}}
\def\F{{\mathbb F}}

\def\P{{\mathbb P}}
\def\E{{\mathbb E}}

\def\G{{\mathbb G}}

\def\Ac{{\mathcal A}}

\def\Fc{{\mathcal F}}

\def\Hc{{\mathcal H}}

\def\Kc{{\mathcal K}}
\def\Lc{{\mathcal L}}
\def\Oc{{\mathcal O}}

\def\Sc{{\mathcal S}}

\def\Uc{{\mathcal U}}
\def\Wc{{\mathcal W}}

\def\gfrak{{\mathfrak g}}

\def\nfrak{{\mathfrak n}}

\numberwithin{equation}{section}
\numberwithin{figure}{section}

\newtheorem{theorem}{Theorem}[section]
\newtheorem{proposition}[theorem]{Proposition}
\newtheorem{corollary}[theorem]{Corollary}

\newtheorem{hypothesis}[theorem]{Hypothesis}

\newtheorem{definition}[theorem]{Definition}
\newtheorem{example}[theorem]{Example}
\newtheorem{lemma}[theorem]{Lemma}

\newtheorem{rmk}[theorem]{Remark}

\newcommand\iid{i.i.d.}

\begin{document}
\maketitle

% --------------------------------------------------------------------
% Abstract
\begin{abstract}
Let $G$ be reductive group over a non-Archimedean local field (e.g $GL_n( \mathbb{Q}_p )$ ) and $G^\vee( \mathbb{C} )$ its  Langlands dual. Jacquet's Whittaker function on $G$ is essentially proportional to the character of an irreducible representation of $G^\vee( \mathbb{C} )$  (a Schur function if $G = GL_n( \mathbb{Q}_p )$). We propose a probabilistic approach to this claim, known as the Shintani-Casselman-Shalika formula, when the group $G$ has at least one minuscule cocharacter.

Thanks to random walks on the group, we start by establishing a Poisson kernel formula for the non-Archimedean Whittaker function. The expression and its ingredients are similar to the one previously obtained by the author in the Archimedean case, hence a unified point of view.
\end{abstract}
\newpage

% --------------------------------------------------------------------
% Table of contents
\setcounter{tocdepth}{4}
\tableofcontents
\newpage

% --------------------------------------------------------------------
\section{Introduction}
\label{section:introduction}

Following Jacquet \cite{Jacquet67}, the Whittaker function is a function on a reductive group over a local field and is defined as an integral. Its relevance to number theory comes from the fact that it plays an important role in proving functional equations for $L$-functions arising from automorphic forms (see for instance Cogdell's lectures \cite{Cogdell}). Here, we are interested in the case of a non-Archimedean field (e.g the $p$-adic field $\Q_p$). We recall in the preliminaries how non-Archimedean Whittaker functions appear in the study of representations from the unramified principal series, as well as some of its basic properties. 

The non-Archimedean Whittaker function is a function on the group that essentially reduces to a function on the coweight lattice, and vanishes if the coweight is not dominant. A first striking result due to Shintani \cite{Shintani76} is that the Whittaker function for $GL_n\left( \Q_p \right)$ is essentially a Schur function for a partition with at most $n$ rows. One has simply to interpret the dominant coweight as a partition. Shintani's proof relies on the crucial fact that the (renormalised) $p$-adic Whittaker function satisfies the same difference equations as Schur functions - namely the Pieri rule. In a way, he does not invoke the integral expression at all. Later, Casselman and Shalika \cite{CS} generalized the result to a reductive group $G$ by proving it is proportional to a character of the Langlands dual $G^\vee\left( \C \right)$. Their proof on the other hand heavily uses integral transforms. 

In this paper, we aim at giving a probabilistic understanding of the Shintani-Casselman-Shalika formula (SCS later in the text), in a similar fashion to the description of the Archimedean Whittaker function in \cite{thesis} and \cite{chh14c}. Although the objects at hand are different, we will use similar notations to the Archimedean case. The approach is original and has many advantages, finding its source in probabilistic potential theory. The general idea is that the SCS formula holds because of the harmonic properties of Whittaker functions, rather than because of their definition as an integral. In fact, our approach is the natural extension of Shintani's work if one considers a convolution in a totally disconnected group as the application of a difference operator. We push the investigation further by analyzing the underlying random walks, for which the non-Archimedean Whittaker functions are harmonic. Similarly, in \cite{thesis} and \cite{chh14c}, Archimedean Whittaker functions are constructed as harmonic functions for a certain hypoelliptic Brownian motion on $G(\C)$ and it is now understood that they are, in a sense, characters associated to $G^\vee$-geometric crystals. In both cases, the starting point is a Poisson kernel formula established thanks to random walks. This way, we give a unified point of view, while emphasising why Whittaker functions on $G$ are characters for $G^\vee$, the Langlands dual.

We will work under the hypothesis that the group $G$ has a minuscule cocharacter, which allows a particularly convenient framework. In the preliminaries (Section \ref{section:preliminaries}), we set the necessary group theoretic notations and introduce two random walks. Both depend on a parameter $z$ which can be interpreted as a drift from the point of view of probability theory and a Satake parameter from the point of view of representation theory. $W^{(z)}$ is a random walk on the cocharacter lattice and $\left( B_t\left( W^{(z)} \right) ; t \in \N \right)$ is a random walk on the Borel subgroup $B$, driven by $W^{(z)}$. The precise definition is given by equations \eqref{eq:B_def_1} and \eqref{eq:B_def_2}.

The first main Theorem \ref{thm:main_result1} can be stated independently, and in fact does not require any knowledge of Whittaker functions. It is a probabilistic characterization of functions on the group $G$ that are $G^\vee$-characters. The theorem claims that, under the appropriate growth conditions, if we consider $f_z$ to be any function that is harmonic for the random walk $\left( B_t\left( W^{(z)} \right) ; t \in \N \right)$ and with the same invariance properties as the Whittaker function, then $f_z$ \emph{has to be} a $G^\vee$-character. The proof is broken down into three steps. 
As announced before, there is a Poisson kernel formula for any such $f_z$ (Proposition \ref{proposition:poisson_formula}). A simple manipulation akin to tropicalization relates $f_z$ to the probability that $W^{(z)}$ remains in the Weyl chamber (Proposition \ref{proposition:local2survival}). Finally, this latter probability is evaluated by implementing a version of the classical reflection principle which forces the appearance of the Weyl character formula, hence giving a $G^\vee$-character (Proposition \ref{proposition:probabilistic_char}). The reflection principle is explained separately in section \ref{section:reflection_principle}. 

The second main Theorem \ref{thm:main_result2} claims that the actual Whittaker function satisfies the harmonicity hypothesis of Theorem \ref{thm:main_result1}, modulo a renormalisation by the modular character. The other hypotheses are easier to establish. Of course, the combination of the two yields the SCS formula (Theorem \ref{thm:CS_formula}). This second main theorem is the focus of section \ref{section:harmonicity_properties}. 
There, we will see that the random walk $\left( B_t\left( W^{(z)} \right) ; t \in \N \right)$ we defined explicitly is not ad hoc: it is obtained by only considering the Borel part of a spherical random walk and penalizing it by the value of its ``diagonal'' part. We called such a random walk spherical because its increments have a distribution that can be understood as elements from the spherical Hecke algebra. We will require knowledge of the Satake isomorphism (subsection \ref{subsection:satake}) and geometric information about double strata from the Cartan and Iwasawa decompositions. It is well-known that this information is encoded in Macdonald's spherical function, which we briefly review in subsection \ref{subsection:macdonald}.

We hope that this paper will serve as a bridge between usually disjoint fields of mathematics and a stepping stone for future work.
%In section \ref{section:concluding_remarks}, we discuss the restrictive hypotheses we made, and hint to future investigations.

% --------------------------------------------------------------------
\section{Setting}
\label{section:preliminaries}

Let $\Kc$ be a non-Archimedean local field i.e a field with a discrete valuation denoted by $\val: \Kc \rightarrow \Z$. We assume further that it is complete and locally compact. A standard reference for the properties of such fields is \cite{Cassels}. Write $\Oc \subset \Kc$ for the ring of integers. We make the choice of a uniformizing element $\varpi \in \left\{ x \in \Kc \ | \ \val(x) = 1 \right\}$. The residue field $\Oc / \varpi \Oc$ is finite with cardinality $q \geq 2$. The absolute value $|\cdot|$ on $\Kc$ is given by $|x| = q^{ -\val(x)}$ for any $x \in \Kc$. The two examples to have in mind are finite extensions of the $p$-adic field $\Q_p$ and the field $\F_q((T))$ of formal Laurent series over the finite field $\F_q$.

If $M$ is a Chevalley group with coefficients in $\Kc$ or a homogenous space, we will denote by $\Fc\left( M \right)$ the $\C$-vector space of locally constant functions on $M$, with compact support.

\subsection{Lie theory}
As usual the multiplicative and additive groups are respectively written $\G_m$ and $\G_a$. Let ${\bf G}$ be a split reductive group (scheme) over $\Z$, as constructed by Chevalley (see for e.g \cite{Steinberg}). We fix a maximal torus ${ \bf T} \approx \left(\G_m\right)^r$, where $r$ is the rank of the group. $G := {\bf G}\left( \Kc \right)$ is its set of $\Kc$-points and $T = {\bf T}\left( \Kc \right) \approx \left(\Kc^*\right)^r$. Let $\left( X, \Phi, X^\vee, \Phi^\vee \right)$ be the associated root datum, meaning that $\Phi$ (resp. $\Phi^\vee$) are the roots (resp. coroots) and the lattice of algebraic characters (resp. cocharacters) of ${ \bf T}$ is $X$ (resp. $X^\vee$):
$$ X = \Hom\left( { \bf T}, \G_m \right) \quad \textrm{ and } \quad X^\vee = \Hom\left( \G_m, { \bf T} \right).$$
We will use an additive notation for the group operation in the $\Z$-modules $X$ and $X^\vee$. Hence we will favor the exponential notation where the image of $t \in T$ by a character $\mu$ is $t^\mu$ and the image of $k \in \G_m$ by $\mu^\vee \in X^\vee$ is $k^{\mu^\vee}$. The canonical pairing between $X$ and $X^\vee$ is denoted by $\langle \cdot, \cdot \rangle$ and is obtained by composing a character $\mu$ and a cocharacter $\lambda^\vee$. It yields $\mu \circ \lambda^\vee: k \mapsto k^{ \langle \mu, \lambda^\vee \rangle } \in \Hom\left( \G_m, \G_m \right) \approx \Z$. The pairing is extended by bilinearity from $X \times X^\vee$ to $X \otimes_\Z \C \times X^\vee \otimes_\Z \C$.

Inside of $\Phi$, we make the choice of a set of positive roots $\Phi^+$ and therefore of a simple root system $\Delta \subset \Phi^+ \subset \Phi$. Also $Q$ is the root lattice spanned by $\Phi$ and $Q \subset P$, where $P$ the weight lattice in duality with the coroot lattice $Q^\vee$. Let $W$ be the Weyl group acting on both $X$ and $X^\vee$. When acting on $X$, $W$ is generated as a Coxeter group by the simple reflections $\left( s_\alpha \right)_{\alpha \in \Delta}$:
$$ \forall x \in X, s_\alpha(x) := x - \langle \alpha^\vee, x \rangle \alpha \ .$$
This action is linearly extended to $X \otimes_\Z \R$ and its fundamental domain is the Weyl chamber:
$$ C := \left\{ x \in X \otimes_\Z \R \ | \ \forall \alpha^\vee \in \Delta, \langle x, \alpha^\vee \rangle \geq 0 \right\}$$
Its interior is denoted by $\overset{\circ}{C}$. The Weyl vector is $ \rho = \half \sum_{\beta \in \Phi^+} \beta$. The length function $\ell: W \rightarrow \N$ gives the minimal number of reflections needed to write an element $w \in W$ as a product of simple reflections. The longest element in $W$ is denoted by $w_0$.

We use similar notations for coweights. The symbols $C^\vee$, $\rho^\vee$, $Q^\vee$ and $P^\vee$ are defined in the same fashion. It is well-known that
$$ Q \subset X \subset P \Leftrightarrow Q^\vee \subset X^\vee \subset P^\vee \ .$$
Such a notation is standard, although perhaps confusing because $P^\vee$ is the dual of $Q$ while $Q^\vee$ is the dual of $P$. The isogeny type of $G$ is determined by $X / Q \subset P / Q$. The Langlands dual $G^\vee\left( \C \right)$ is the connected reductive group over $\C$ whose root datum is $\left( X^\vee, \Phi^\vee, X, \Phi \right)$. The highest weight $G^\vee$-module with highest weight $\lambda^\vee \in \left(X^\vee\right)^+$ is written $V\left( \lambda^\vee \right)$ and its character is denoted by $\ch V\left( \lambda^\vee \right) \in \C\left[ X^\vee \right]^W$.

At the level of the group, our choice of positive roots gives a pair of opposite Borel subgroups $\left( B, B^+ \right)$. $N$ is the unipotent radical of the lower Borel subgroup $B$, while $U$ is the unipotent radical of the upper Borel subgroup $B^+$, hence $B = NT$ and $B^+ = TU$. We will mainly work in the lower Borel subgroup $B$ and its unipotent subgroup $N$. For every root $\beta \in \Phi$, we have a group homomorphism $x_\beta: \G_a \rightarrow G$ coming from a pinning and the group $G$ is generated by the subgroups $\left( x_\beta\left( \Kc \right) \right)_{\left( \beta \in \Phi \right)}$ (\textsection 3 in \cite{Steinberg}). We fix this pinning throughout the paper.

The image of the uniformizing element $\varpi \in \Oc$ by the cocharacter $\mu^\vee \in X^\vee$ is denoted by $\varpi^{\mu^\vee} \in T$. This way, one has an embedding of $X^\vee$ into $T$ with image
$$ A := \bigsqcup_{\mu^\vee \in X^\vee} \varpi^{-\mu^\vee} \ .$$

Let $K = { \bf G}(\Oc)$ be the special maximal compact open group of integral points (\cite{Tits} 3.8.1). It enters in various decompositions that we will often use and that are established thanks to the Bruhat-Tits theory of buildings \cite{Tits}. Here are the non-Archimedean analogues of respectively the singular value decomposition and the Gram-Schmidt decomposition, if we were in the context of the general linear group over $\R$ or $\C$:

\begin{align}
\label{eq:cartan_decomposition} 
\textrm{ (Cartan decomposition)  } & G = \bigsqcup_{\lambda^\vee \in \left( X^{\vee} \right)^+} K \varpi^{-\lambda^\vee} K \ , \\
\label{eq:iwasawa_decomposition} 
\textrm{ (Iwasawa decomposition) } & G = \bigsqcup_{\mu^\vee \in X^{\vee}} N \varpi^{-\mu^\vee} K = NAK \ .
\end{align}

The groups $T\left( \Oc \right)$, $B\left( \Oc \right)$ and $N\left( \Oc \right)$ correspond to the respective subgroups of $T$, $B$ and $N$ obtained by intersecting with $K$. Of course, because we are dealing with split groups, these are integral points of the corresponding algebraic groups (e.g $T\left( \Oc \right) = {\bf T}\left( \Oc \right)$).

For the convenience of the reader, let us recall the definition of a minuscule coweight:
\begin{definition}[\cite{Bourbaki}, Chapitre VIII, \textsection 7, Proposition 6]
If the root datum is irreducible, a dominant coweight $\lambda^\vee \in X^\vee$ is called minuscule if it is non-zero and satisfies one of the equivalent conditions:
\begin{itemize}
 \item The set of weights in the $G^\vee$-module $V\left( \lambda^\vee \right)$ is the Weyl group orbit $W \lambda^\vee$ and each weight has multiplicity one.
 \item For all $\alpha \in \Delta$ and $w \in W$, $\langle \alpha, w \lambda^\vee \rangle \in \{-1, 0, 1\}$.
\end{itemize}
\end{definition}

In the general reducible case, we will consider a coweight to be minuscule if it has a minuscule component in every irreducible factor. As announced in the introduction, we make the following hypothesis:
\begin{hypothesis}
We assume that $G$ has a minuscule cocharacter, denoted by $\Lambda^\vee \in X^\vee$ and fixed once and for all.
\end{hypothesis}

\begin{rmk}
Let $\gfrak$ be the Lie algebra of ${\bf G}$. The group is reductive if the derived algebra $[\gfrak, \gfrak]$ of $\gfrak$ is a semisimple algebra. Necessarily, our hypothesis implies that this semisimple algebra is the direct sum of simple Lie algebras of type $A_n$, $B_n$, $C_n$, $D_n$, $E_6$ or $E_7$. These are the only ones having a minuscule coweight (and weight) in the Cartan-Killing classification.

The hypothesis also imposes restrictions on the isogeny type of ${\bf G}$, as we demand that the minuscule coweight $\Lambda^\vee$ gives a cocharacter. In order that coweights of the Lie algebra lift to cocharacters of ${\bf G}$, we can assume that ${\bf G}$ is of adjoint type ($X = Q$) or equivalently that $G^\vee$ is simply connected ($X^\vee = P^\vee$). The author believes that such isogeny restrictions can be lifted. Getting rid of the minuscule hypothesis will however require a different setup.

For the sake of concreteness, it is sufficient to assume that the semisimple part of ${\bf G}$ is a direct product of simple adjoint groups, each factor from the above type (see exercise p.46 in \cite{Steinberg}). In this case, the coweight $\Lambda^\vee$ has a minuscule component on every simple factor, according to the list at the end of \cite{Bourbaki}, Chapitre VIII, \textsection 7.3. 
\end{rmk}

\subsection{Measures and probability} Equality in law, i.e equality between the probability distributions of two random variables $X$ and $Y$, is denoted by $X \eqlaw Y$. On any compact group $H$, we denote by $\Uc\left( H \right)$ a generic Haar distributed random variable on $H$. For example $U \eqlaw \Uc\left( \Oc \right)$ is a Haar distributed random variable on the ring of integers $\Oc$. 

For the non-compact groups $G$ and $N$, each Haar measure is normalized so that $K$ and $K \cap N = N\left( \Oc \right)$ have unit measure.

Also, a left-invariant random walk on a group $H$ is a Markov chain $\left( H_t ; t \in \N \right)$ such that:
$$ \forall t \in \N, \ H_t := h_1 h_2 \dots h_t \ .$$
where $\left( h_s \right)_{s \in \N^*}$ are independent and identically distributed random variables, called the increments. A function $f: H \rightarrow \C$ is $\alpha$-harmonic (w.r.t that random walk) if 
$$ \forall h \in H, \ \forall t \in \N, \ \E\left[ f\left( h H_t \right) \right] = \alpha^t f\left( h \right) \ .$$
If $\alpha=1$, one simply says that $f$ is harmonic.

\subsection{Unramified principal series}

Let $\chi: T \rightarrow \C^*$ be a multiplicative character of $T$, the Satake parameter. We assume it is unramified i.e trivial on $T\left( \Oc \right) = T \cap K$. Necessarily, it is entirely determined by a $z \in X \otimes_\Z \C$, which is unique modulo $2 \pi i \cdot X$ and such that:
\begin{align}
\label{eq:def_chi_z}
\forall \mu^\vee \in X^\vee, \ \chi\left( \varpi^{-\mu^\vee} \right) & = e^{ \langle z, \mu^\vee \rangle } \ .
\end{align}
In all the following, we loosely identify functions of $\chi$ and functions of $z \textrm{ mod } 2 \pi i \cdot X$ because of the one-to-one correspondence between them. By abuse of language, both of them will be called the Satake parameter. By setting $\chi$ to be trivial on $N$, the character $\chi$ is inflated to $B$. Now consider the representation $I\left( \chi \right)$ defined as the parabolic induction:
\begin{align*}
I\left( \chi \right) & := \textrm{Ind}_{B}^G \chi \\
                     & = \left\{ f \in \Fc\left( G \right) \ | \ \forall b \in B, \forall g \in G, f(b g) = \chi(b) \delta^\half(b) f(g) \right\} \ .
\end{align*}

Here $\delta: B \rightarrow \C^*$ is the modular character. It is trivial on $N$ and $T \cap K$ and is obtained from the determinant of the $\textrm{Ad}$ action on $\nfrak$, the Lie algebra of $N$. More precisely, for $\varpi^{\mu^\vee} \in T / T \cap K$, it is given by:
\begin{align}
\label{eq:def_modular}
\delta\left( \varpi^{\mu^\vee} \right) & = \left| \det\left( \textrm{Ad}\left( \varpi^{\mu^\vee} \right)_{|\nfrak} \right) \right|
                                         = \left| \varpi^{-\langle 2 \rho, \mu^\vee \rangle} \right|
                                         = q^{ {\langle 2 \rho, \mu^\vee \rangle} } \ .
\end{align}
Inside the representation $I\left( \chi \right)$, there is a unique vector (up to a scalar) $\Phi_\chi$ that is $K$ right-invariant. This is a simple consequence of the $G=NAK$ decomposition:
$$ \forall \left( n, a, k \right) \in N \times A \times K, 
   \Phi_\chi\left( nak \right) = \chi(a) \delta^\half(a) \Phi_\chi\left( \Id \right) \ .$$
We choose the usual normalization given by $\Phi_\chi\left( \Id \right) = 1$. Moreover, in the following, we will rather consider the vector $\Phi_{\chi^{w_0}}$ inside the representation $I\left( \chi^{w_0} \right)$. Here $\chi^{w_0}$ is the character with Satake parameter $w_0 z$:
\begin{align}
\label{eq:def_chi_w0_z}
\forall \mu^\vee \in X^\vee, \ \chi^{w_0}\left( \varpi^{-\mu^\vee} \right) & = e^{ \langle w_0 z, \mu^\vee \rangle } \ .
\end{align}
Ultimately, the reason of this choice stems from the domain of absolute convergence we want to impose on the Whittaker function.

\subsection{Jacquet's Whittaker function (the non-Archimedean case)}

Let $\psi: ( \Kc,+) \rightarrow ( \C^*, \times)$ be an additive character of $\Kc$, trivial on the ring of integers $\Oc$, but non-trivial on $\varpi^{-1} \Oc$. One obtains a character
\begin{align}
\label{eq:def_psi_N} 
\varphi_N := & \psi \circ \chi_{st}^-: N \rightarrow \C^*
\end{align}
by composing $\psi$ with the standard character of $N$:
$$ \chi_{st}^- = \sum_{\alpha \in \Delta} \chi_\alpha^- \ ,$$
where $\chi_\alpha^-\left( x_{-\beta}(t) \right) = t \mathds{1}_{\{ \alpha = \beta \}}$. Jacquet's Whittaker function associated to the Satake parameter $z \in X \otimes_\Z \C$ and the character $\varphi_N$ is by definition (\cite{Jacquet67}):
\begin{align}
\label{eq:whittaker_def}
\forall g \in G, \Wc_z\left( g \right) = \Wc_\chi\left( g \right) := & \int_N \Phi_{ \chi^{w_0} }( \bar{w}_0 n g ) \varphi_N(n)^{-1} dn
\end{align}
where $\bar{w}_0$ is any choice of representative for the longest element in the Weyl group. We will invariably use the subscript $z$ or $\chi$ depending on what is more convenient. Also, because the dependence in the character $\varphi_N$ is irrelevant, there will be no subscript to emphasize it.

The integral \ref{eq:whittaker_def} does not converge for all $z \in X \otimes_\Z \C$, but can be analytically extended. That fact is particularly obvious when looking at the SCS formula. For now, let us list the following well-known properties. These hold only within the region of convergence of the Jacquet integral \eqref{eq:whittaker_def} and not for its analytic continuation.

\begin{proposition}
\hfill
\begin{itemize}
 \item (Domain of definition) The integral in equation \eqref{eq:whittaker_def} is absolutely convergent for $\Re\left( z \right) \in \overset{\circ}{C}$.
 \item (Boundedness) $\delta^{-\half} \chi^{-1} \Wc_\chi$ is bounded.
 \item (Invariance property)
	\begin{align}
	\label{eq:torsion}
	\forall \left(n, g, k\right) \in N \times G \times K, \ \Wc_\chi\left( n g k \right) = \varphi_N\left( n \right) \Wc_\chi\left( g \right) \ ,
	\end{align}
	and hence $\Wc_\chi$ reduces to a function on $A$.
 \item (Asymptotic behavior)
       There is a function $c_G: X \otimes_\Z \C \rightarrow \C$ given explicitly by the Gindikin-Karpelevich formula (Theorem \ref{thm:GK_formula}) such that, as $\lambda^\vee \rightarrow \infty$ in the dual Weyl chamber:
       \begin{align}
       \label{eq:whittaker_asymptotics}
       \lim_{\lambda^\vee \rightarrow \infty} \left( \delta^{-\half} \chi^{-1} \Wc_\chi \right)\left( \varpi^{-\lambda^\vee} \right) & = c_G\left( z \right) \ .
       \end{align}
\end{itemize}
\label{proposition:whittaker_properties}
\end{proposition}
\begin{rmk}
 \label{rmk:infinity_in_C}
 When considering limits $\lambda^\vee \rightarrow \infty$ in the dual Weyl chamber, we always mean that $\lambda^\vee$ goes to infinity away from the walls. Formally:
 $$ \forall \alpha \in \Delta, \ \langle \alpha, \lambda^\vee \rangle \rightarrow \infty$$
\end{rmk}

%\begin{align}
%\label{eq:whittaker_dominance}
%\forall g \in G, \left| \Wc_\chi\left( g \right) \right| & \leq \int_N \Phi_{|\chi|}( n g ) dn < \infty
%\end{align}
\begin{proof}
For now, let us suppose that $\chi$ is chosen so that the integral defining $\Wc_\chi$ converges absolutely. Equation \eqref{eq:torsion} is proven by performing a change of variable in the integral formula: 
\begin{align*}
\Wc_\chi(ngk) & = \int_N \Phi_{ \chi^{w_0} }( \bar{w}_0 h n g k ) \varphi_N(h)^{-1} dh\\
& = \int_N \Phi_{ \chi^{w_0} }( \bar{w}_0 h g ) \varphi_N(h n^{-1})^{-1} dh\\
& = \varphi_N(n) \Wc_\chi(g) \ .
\end{align*}
Because of the Iwasawa decomposition, we can focus on $g \in A$. In this case, we start by proving the alternative expression:
\begin{align}
\label{eq:alternate_expr}
\forall a \in A, \ \Wc_\chi(a) & = \chi\left( a \right) \delta^{\half}\left( a \right) \int_N \Phi_{ \chi^{w_0} }( \bar{w}_0 n) \varphi_N\left( a n a^{-1} \right)^{-1} d n \ .
\end{align}
This latter equation is obtained as follows:
\begin{align*}
\Wc_\chi(a) & = \int_N \Phi_{ \chi^{w_0} }( \bar{w}_0 n a ) \varphi_N(n)^{-1} dn\\
& = \Phi_{\chi^{w_0}}\left( a^{w_0} \right) \int_N \Phi_{ \chi^{w_0} }( \bar{w}_0 a^{-1} n a ) \varphi_N(n)^{-1} dn\\
& = \chi\left( a \right) \delta^{-\half}\left( a \right) \int_N \Phi_{ \chi^{w_0} }( \bar{w}_0 n) \varphi_N\left( a n a^{-1} \right)^{-1} d \left( a n a^{-1} \right)\\
& = \chi\left( a \right) \delta^{\half}\left( a \right) \int_N \Phi_{ \chi^{w_0} }( \bar{w}_0 n) \varphi_N\left( a n a^{-1} \right)^{-1} d n \ .
\end{align*}
Then, as the character $\varphi_N$ is unitary i.e valued in $S^1 = \left\{ s \in \C \ | \ |s|=1 \right\}$, we see that for all $a \in A$:
$$\left| \delta^{-\half} \chi^{-1} \Wc_\chi \right|(a)
= \left| \int_N \Phi_{ \chi^{w_0} }( \bar{w}_0 n) \varphi_N\left( a n a^{-1} \right)^{-1} d n \right|
\leq \int_N \Phi_{ |\chi^{w_0}| }( \bar{w}_0 n) d n
$$
where $|\chi|$ is the real valued character such that for all $\mu^\vee \in X^\vee$, $|\chi|\left( \varpi^{-\mu^\vee} \right) = e^{\langle \Re(z), \mu^\vee \rangle}$. Hence, thanks to Theorem \ref{thm:GK_formula}, the Whittaker function is indeed well-defined for $\Re\left( z \right) \in \overset{\circ}{C}$. Moreover, in such a case, $\delta^{-\half} \chi^{-1} \Wc_\chi$ is bounded on $A$ and a fortiori on the entire group, because of the invariance property.

Now, when considering $a = \varpi^{-\lambda^\vee}$ in equation \eqref{eq:alternate_expr}, we have:
$$ \left( \delta^{-\half} \chi^{-1} \Wc_\chi \right)\left( \varpi^{-\lambda^\vee} \right)
 = \int_N \Phi_{ \chi^{w_0} }( \bar{w}_0 n) \varphi_N\left( \varpi^{-\lambda^\vee} n \varpi^{\lambda^\vee} \right)^{-1} dn \ .$$
If $\lambda^\vee \rightarrow \infty$ while staying in the Weyl chamber $C^\vee$ and away from the walls, the matrix coefficients of $\varpi^{-\lambda^\vee} n \varpi^{\lambda^\vee}$ become $\Oc$ valued. Hence, for every $n \in N$, $\varphi_N\left( \varpi^{-\lambda^\vee} n \varpi^{\lambda^\vee} \right)^{-1}$ is stationary at $1$ and using Lebesgue's dominated convergence theorem:
$$
\left( \delta^{-\half} \chi^{-1} \Wc_\chi \right)\left( \varpi^{-\lambda^\vee} \right) \stackrel{\lambda \rightarrow \infty}{\longrightarrow} \int_N \Phi_{ \chi^{w_0} }( \bar{w}_0 n  ) dn \ .
$$
This is equation \eqref{eq:whittaker_asymptotics}.
\end{proof}

The function $c_G$ is the non-Archimedean analogue of the Harish-Chandra $c$ function and appears as:
\begin{theorem}[Langlands - \cite{Langlands71} formula (4)]
\label{thm:GK_formula}
When $\Re(z) \in \overset{\circ}{C}$, we have:
$$ \int_N \Phi_{ |\chi^{w_0}| }( \bar{w}_0 n) d n < \infty \ ,$$
and the following Gindikin-Karpelevich formula holds for the non-Archimedean places:
$$ c_{G}\left( z \right)
:= \int_N \Phi_{ \chi^{w_0} }\left( \bar{w}_0 n \right) dn
 = \prod_{\beta \in \Phi^+}\left( \frac{1 - q^{-1} e^{-\langle \beta^\vee, z \rangle} }
                                                         {1 -        e^{-\langle \beta^\vee, z \rangle} } \right) \ .$$
\end{theorem}

\subsection{Random walks} For the purpose of constructing probability measures, in all the following, we will assume that the Satake parameter $z$ is real and in the interior of the Weyl chamber $\overset{\circ}{C}$. Let $\left( \Omega, \Ac, \P \right)$ be our working probability space. We consider a random walk $W^{(z)} = \left( W_t^{(z)}; t \in \N \right)$ on $X^\vee$, starting at zero, with independent and identically distributed increments. We choose the following distribution on the Weyl group orbit of $\Lambda^\vee$:
\begin{align}
\label{eq:rw_lattice}
\forall t \in \N, \forall \mu^\vee \in W \Lambda^\vee, \ \P\left( W_{t+1}^{(z)} - W_t^{(z)} = \mu^\vee \right) & = \frac{ e^{ \langle z, \mu^\vee \rangle } }{ \ch V\left( \Lambda^\vee \right)(z) } \ .
\end{align}
Because $\Lambda^\vee$ is minuscule, the weights $W \Lambda^\vee$ are exactly the weights appearing in the $G^\vee$-module  $V\left( \Lambda^\vee \right)$ and hence:
$$ \ch V\left( \Lambda^\vee \right) = \sum_{\mu^\vee \in W \Lambda^\vee} e^{ \mu^\vee } \ .$$
Therefore, the measure defined in \eqref{eq:rw_lattice} sums indeed to $1$ and it is a probability distribution of representation-theoretic significance. Notice that the random walk $W^{(z)}$ depends on the choice of minuscule coweight $\Lambda^\vee$. Also, the Satake parameter determines the drift, and is indicated in superscript notation. An important feature is that, almost surely, the random walk $W^{(z)}$ goes to infinity inside $C^\vee$. This is obtained thanks to the law of large numbers and the following lemma:
\begin{lemma}
For $z \in \overset{\circ}{C}$, the random walk's mean drift $\E\left( W_1^{(z)} - W_0^{(z)} \right)$ is in the interior of the dual Weyl chamber $C^\vee$.
\end{lemma}
\begin{proof}
For each simple root $\alpha \in \Delta$, we have that
 $ \langle \alpha, \E\left( W_1^{(z)} - W_0^{(z)} \right) \rangle
 = \frac{ \sum_{\mu^\vee \in W \Lambda^\vee} \langle \alpha, \mu^\vee \rangle e^{ \langle z, \mu^\vee \rangle } }
        { \ch V\left( \Lambda^\vee \right)(z) } \ ,$
which has the same sign as its numerator. By doubling the sum and summing over both $\mu^\vee$ and $s_\alpha \mu^\vee$, the latter numerator has the same sign as:
$$ \sum_{\mu^\vee \in W \Lambda^\vee} \langle \alpha, \mu^\vee \rangle e^{ \langle z, \mu^\vee \rangle }
 + \sum_{\mu^\vee \in W \Lambda^\vee} \langle \alpha, s_\alpha \mu^\vee \rangle e^{ \langle z, s_\alpha \mu^\vee \rangle }
 = \sum_{\mu^\vee \in W \Lambda^\vee} \langle \alpha, \mu^\vee \rangle e^{ \langle z, \mu^\vee \rangle } 
                                      \left( 1 - e^{-\langle z, \alpha^\vee \rangle \langle \alpha, \mu^\vee \rangle} \right)$$
Because $\langle z, \alpha^\vee \rangle > 0$, the function $t \mapsto t \left( 1 - e^{-\langle z, \alpha^\vee \rangle t } \right)$ maps $\R$ to $\R_+$ and one recognizes a sum of non-negative terms. At least one term is positive because of the minuscule hypothesis, and hence $\langle \alpha, \E\left( W_1^{(z)} - W_0^{(z)} \right) \rangle > 0$, concluding the proof.
\end{proof}

Now, we define $\left( B_t\left( W^{(z)} \right) ; t \in \N \right)$ as a left invariant random walk on $B\left( \Kc \right)$ with independent increments. More precisely:
\begin{align}
\label{eq:B_def_1}
B_{0}\left( W^{(z)} \right) & = b_0 \ ,\\
\label{eq:B_def_2}
\forall t \in \N, \ 
B_{t+1}\left( W^{(z)} \right) & = B_{t}\left( W^{(z)} \right) b_{t+1}' \varpi^{-\left( W^{(z)}_{t+1} - W^{(z)}_{t} \right) } b_{t+1} \ ,
\end{align}
where $\left( b_t; t \in \N \right)$ and $\left( b_t'; t \in \N \right)$ are all independent and distributed as $\Uc\left( B\left( \Oc \right) \right)$. Notice that the distribution of $B_t\left( W^{(z)} \right)$ is bi-invariant under $B\left( \Oc \right)$.

Modulo $T\left( \Oc \right) = T \cap K$ on the right, we have the $NA$ decomposition:
\begin{align}
\label{eq:NA_decomposition}
B_{t}\left( W^{(z)} \right) & = N_{t}\left( W^{(z)} \right) A_{t}\left( W^{(z)} \right) \mod T\left( \Oc \right) \ ,
\end{align}
with, for all $t \in \N$,
\begin{align}
\label{eq:def_A_t}
A_{t}\left( W^{(z)} \right) & = \varpi^{ -W^{(z)}_t } \ ,
\end{align}
and
\begin{align}
\label{eq:def_N_t_1}
N_{0}\left( W^{(z)} \right) & = n_0 \ ,\\
\label{eq:def_N_t_2}
N_{t+1}\left( W^{(z)} \right) & = N_{t}\left( W^{(z)} \right) \varpi^{-W^{(z)}_{t}  } n_{t+1}' \varpi^{W^{(z)}_{t}}
                                                              \varpi^{-W^{(z)}_{t+1}} n_{t+1}  \varpi^{W^{(z)}_{t+1}} \ .
\end{align}
Here, $\left(n_t ; t \in \N \right)$ and $\left(n_t' ; t \in \N \right)$ are all independent and distributed as $\Uc\left( N\left( \Oc \right) \right)$. These are obtained as the $N$-parts of $\left( b_t; t \in \N \right)$ and $\left( b_t'; t \in \N \right)$.

\begin{proposition}
\label{proposition:N_infty_CV}
Almost surely, the process $$\left( N_{t}\left( W^{(z)} \right) ; t \in \N \right)$$ converges to a limiting random variable $N_\infty\left( W^{(z)} \right)$.
\end{proposition}
\begin{proof}
The formal proof is given in section \ref{section:N_infty_CV}. The result is a consequence of the deterministic Theorem \ref{thm:CV}, which holds as soon as the increments tend to $\Id$. The hypothesis is satisfied for $\left( N_{t}\left( W^{(z)} \right) ; t \in \N \right)$ because of equation \eqref{eq:def_N_t_2} and the fact that almost surely:
$$ \forall n \in N, \ \varpi^{-W^{(z)}_t} n \varpi^{ W^{(z)}_t} \stackrel{t \rightarrow \infty}{\longrightarrow} \Id,$$
as $W^{(z)}_t \rightarrow \infty$ in the dual Weyl chamber. The convergence is uniform in $n$ on compact sets and $N\left( \Oc \right)$ is compact.
\end{proof}

The reader might find example \ref{example:PGL2} more enlightening.
\begin{example}[$G = PGL_2\left( \Kc \right)$]
\label{example:PGL2}
The only adjoint group of rank $1$ is $PGL_2\left( \Kc \right)$. It is obtained as the quotient of $GL_2\left( \Kc \right)$ by the central $\Kc^* \Id$. The Borel subgroup $B$ is made of lower triangular matrices modulo $\Kc^* \Id$ and we choose as representatives matrices of the form:
$$ B = \left\{ \begin{pmatrix} a & 0\\ b & 1 \end{pmatrix} , a \in \Kc^*, b \in \Kc \right\} \ .$$
For the character and cocharacter lattices:
\begin{align*}
X      = \Z \alpha_1,      & \quad \quad \alpha_1: \begin{pmatrix} k_1 & 0\\ 0 & k_2 \end{pmatrix} \mapsto k_1/k_2 \ ,\\
X^\vee = \Z \omega_1^\vee, & \quad \quad \omega_1^\vee: k \mapsto  \begin{pmatrix} k & 0\\ 0 & 1 \end{pmatrix} \ .
\end{align*}
We have $X \otimes_\Z \C \approx \C$ and the Weyl chamber is the positive real line $\R_+$ in this identification. As such, the parameter $z$ is a positive real number. The random walk $\left( W_t^{(z)} ; t \in \N \right)$ is the simple random walk on $X^\vee \approx \Z$ with:
$$ \P\left( W_{t+1}^{(z)} - W_t^{(z)} = 1 \right) = 1 - \P\left( W_{t+1}^{(z)} - W_t^{(z)} = -1 \right) = \frac{ e^{z} }{ e^{z} + e^{-z} } > \half,$$
and therefore $W^{(z)}$ drifts to $\infty$. Moreover:
$$ A_{t}\left( W^{(z)} \right) = \begin{pmatrix} \varpi^{-W_t^{(z)}} & 0\\ 0 & 1 \end{pmatrix} \ ,$$
and there are independent Haar distributed random variables $\left(c_t ; t \in \N \right)$ such that:
$$ N_{t}\left( W^{(z)} \right) = \begin{pmatrix} 1 & 0\\ \sum_{s=0}^t c_s \varpi^{W_s^{(z)}} & 1 \end{pmatrix} \ .$$
Because $W^{(z)}_s \stackrel{s \rightarrow \infty}{\longrightarrow} \infty$, the general term $c_s \varpi^{W_s^{(z)}}$ goes to zero as $s \rightarrow \infty$. Now, invoking the standard fact that a series in a complete ultrametric space converges as soon as its general term goes to zero (Lemma 1.3 in chapter 4 \cite{Cassels}), we have that $\sum c_s \varpi^{W_s^{(z)}}$ converges in the $\varpi$-adic topology. Therefore, the sequence $\left( N_{t}\left( W^{(z)} \right); t = 0,1,2, \dots \right)$ converges almost surely to
$$ N_{\infty}\left( W^{(z)} \right) = \begin{pmatrix} 1 & 0\\ \sum_{s=0}^\infty c_s \varpi^{W_s^{(z)}} & 1 \end{pmatrix} \ .$$
\end{example}

% --------------------------------------------------------------------
\section{Main results}
\label{section:main_results}

For $z \in C$, let $b(z)$ be the inverse of the Weyl denominator for the group $G^\vee\left( \C \right)$ (p. 318, eq. (5.71) in \cite{knapp02}), evaluated at $z$:
$$ b(z) := \frac{1}{ e^{\langle z, \rho^\vee \rangle} \prod_{\beta \in \left( \Phi^\vee \right)^+} \left( 1 - e^{-\langle z, \beta^\vee \rangle} \right) },$$
and recall that $\chi := \chi_z: T \rightarrow \C$ is the unramified multiplicative character with Satake parameter $z$ (eq. \eqref{eq:def_chi_z}).

\begin{theorem}
%[Martingale problem associated to a character]
\label{thm:main_result1}
There exists one and only one locally constant function $\psi_\chi: B \rightarrow \C$ such that:
\begin{itemize}
 \item (Harmonicity) $\chi^{-1} \psi_\chi: B \rightarrow \C$ is harmonic for the random walk $\left( B_t\left( W^{(z)} \right) ; t \in \N \right)$ i.e:
\begin{align}
\label{eq:harmonicity_property} 
\forall t \in \N, \ \forall b \in B, \ 
\chi\left( b \right)^{-1} \psi_\chi\left( b \right) & = \E\left[ \left( \chi^{-1} \psi_\chi \right)\left( b B_t\left( W^{(z)} \right) \right) \right] \ .
\end{align}
 \item (Boundedness) $\chi^{-1} \psi_\chi$ is bounded.
 \item (Behavior at infinity) If $\lambda^\vee \rightarrow \infty$ in the interior of the dual Weyl chamber, then:
       \begin{align}
       \label{eq:behavior_at_infty}
       \left( \chi^{-1} \psi_\chi \right)\left( \varpi^{-\lambda^\vee} \right) & \rightarrow b(z) e^{\langle z, \rho^\vee \rangle} \ .
       \end{align}
 \item (Invariance property) For all $\left( n, b, k \right) \in N \times B \times B\left( \Oc \right)$, we have:
       $$ \psi_\chi\left( n b k \right) = \varphi_N(n) \psi_\chi(b) \ .$$
\end{itemize}
Moreover, $\psi_\chi$ is given by a $G^\vee$-character:
\begin{align}
\label{eq:normalized_CS_formula}
\psi_\chi\left( \varpi^{-\lambda^\vee} \right) = 
       \left\{ \begin{array}{cc}
                   \ch V\left( \lambda^\vee \right)(z) & \textrm{ if } \lambda^\vee \textrm{ dominant, } \\
                   0                                   & \textrm{ otherwise. }
               \end{array} \right.
\end{align}
\end{theorem}
\begin{proof}
Uniqueness can be proved separately using the following standard martingale argument. Suppose there are two such functions, and name $g$ their difference. The martingale $\left( \chi^{-1}g \right)\left[ B_t\left( W^{(z)} \right) \right]$ is bounded and goes to zero almost surely in infinite time. Therefore $g\left[ B_t\left( W^{(z)} \right) \right]$ is identically zero.

In order to prove that $\psi_\chi$ is given by a character, one has to invoke successively the propositions \ref{proposition:poisson_formula}, \ref{proposition:local2survival} and \ref{proposition:probabilistic_char}. These are the subject of the next subsection.
\end{proof}

\begin{theorem}
\label{thm:main_result2}
The function on $B$ defined by
$$ b \mapsto \left( \delta^{-\half} \chi^{-1} \Wc_\chi \right)\left( b \right)$$
is harmonic for the random walk $\left( B_t\left( W^{(z)} \right) ; t \in \N \right)$.
\end{theorem}
\begin{proof}
See subsection \ref{subsection:proof_thm_2}.
\end{proof}

Putting together Theorems \ref{thm:main_result1} and \ref{thm:main_result2}, we obtain the:
\begin{theorem}[Shintani-Casselman-Shalika formula \cite{CS}, \cite{Shintani76}]
\label{thm:CS_formula}
\begin{align*}
  & \Wc_z\left( \varpi^{-\lambda^\vee} \right) \\
= & 
       \left\{ \begin{array}{cc}
                   \delta^{\half}\left( \varpi^{-\lambda^\vee} \right) \ch V\left( \lambda^\vee \right)(z)
                   \prod_{\beta \in \Phi^+}\left( 1 - q^{-1} e^{ -\langle \beta^\vee, z \rangle} \right) & \textrm{ if } \lambda^\vee \textrm{ dominant, } \\
                   0                                                                                     & \textrm{ otherwise. }
               \end{array} \right. \nonumber
\end{align*}
\end{theorem}
\begin{proof}
Let $\psi_\chi: B \rightarrow \C$ denote the normalized Whittaker function as opposed to Jacquet's Whittaker function:
$$ \forall b \in B, \psi_\chi\left( b \right) := \frac{ \Wc_\chi\left( b \right) \delta^{-\half}(b) }
                                                      { \prod_{\beta \in \Phi^+}\left( 1 - q^{-1} e^{ -\langle \beta^\vee, z \rangle} \right) } \ .$$
This function satisfies the hypotheses of Theorem \ref{thm:main_result1}: the harmonicity property is Theorem \ref{thm:main_result2} and the other properties come from Proposition \ref{proposition:whittaker_properties}.
\end{proof}

\subsection{A characterisation of characters}
Theorem \ref{thm:main_result1} can be taken independently from the rest and seen as a characterisation of characters. The proof of this characterisation decomposes into three steps.

\begin{proposition}[Poisson kernel formula]
\label{proposition:poisson_formula}
If $\psi_\chi$ satisfies the hypotheses of Theorem \ref{thm:main_result1} then the following Poisson formula holds:
\begin{align}
\label{eq:poisson_formula}
  \psi_\chi\left( \varpi^{-\lambda^\vee} \right)
= b(z) e^{ \langle z, \lambda^\vee + \rho^\vee \rangle }
  \E\left[ \varphi_N\left( \varpi^{-\lambda^\vee} N_\infty\left( W^{(z)} \right) \varpi^{\lambda^\vee} \right) \right] \ .
\end{align}
\end{proposition}
\begin{proof}
Because $f := \chi^{-1} \psi_\chi$ is harmonic for $\left( B_t\left( W^{(z)} \right) ; t \in \N \right)$, we have that for every $b = na \in NA$ and for all $t \in \N$:
$$ f\left( b \right) = \E\left[ f\left( b B_t\left( W^{(z)} \right)  \right) \right] \ .$$
Using the $NA$ decomposition from equation \eqref{eq:NA_decomposition}:
$$b B_t\left( W^{(z)} \right)
= n a N_t\left( W^{(z)} \right) a^{-1} a A_t\left( W^{(z)} \right) \mod T\left( \Oc \right),$$
and the invariance property, we have that for all $t \in \N$:
\begin{align}
\label{eq:intermediate_f}
f\left( na \right) & = \E\left[ \varphi_N\left( n a N_t\left( W^{(z)} \right) a^{-1} \right) f\left(a A_t\left( W^{(z)} \right)\right) \right] \ .
\end{align}
Since $W^{(z)}$ goes to infinity while staying in the interior of $C^\vee$, we have from the prescribed behavior at infinity in equation \eqref{eq:behavior_at_infty}:
$$ f\left(a A_t\left( W^{(z)} \right)\right) = f\left(a \varpi^{ -W^{(z)}_t } \right) \stackrel{t \rightarrow \infty}{\longrightarrow} b(z) e^{ \langle z, \rho^\vee \rangle } \ .$$
Adding to that the convergence of $N_t\left( W^{(z)} \right)$ to $N_\infty\left( W^{(z)} \right)$ (Proposition \ref{proposition:N_infty_CV}), we have the result by passing to the limit under the expectation in equation \eqref{eq:intermediate_f}. Boundedness for $f$ gives the required domination in order to apply Lebesgue's dominated convergence theorem.
\end{proof}

Let $f: D \rightarrow \R$ be a bounded harmonic function on a regular plane domain $D \subset \C$. Recall from classical potential theory that $f$ is entirely determined by its values on the boundary $\partial D$. More precisely, the Poisson formula says that $f$ can be represented as the integral of $f_{|\partial D}$ against the Poisson kernel. By interpreting the distribution of $N_\infty\left( W^{(z)} \right)$ as a Poisson kernel, the probabilistic representation formula \eqref{eq:poisson_formula} is indeed a Poisson formula: the value of the harmonic function $\chi^{-1} \psi_\chi$ prescribed by a measure on the ``boundary'' $N$. This measure is the law of $N_\infty\left( W^{(z)} \right)$. A completely analogous fact holds in the Archimedean case \cite{chh14c}, \cite{thesis}.

Now, we proceed to rearranging the expression in equation \eqref{eq:poisson_formula}.
\begin{proposition}
\label{proposition:local2survival}
\begin{align*}
\E\left[ \varphi_N\left( \varpi^{\lambda^\vee} N_\infty\left( W^{(z)} \right) \varpi^{-\lambda^\vee} \right) \right] & = \P\left( \lambda^\vee + W^{(z)} \textrm{ remains in } C^\vee \right)
\end{align*}
\end{proposition}
Basically, we claim that if we kill the random walk $\lambda^\vee + W^{(z)}$ upon exiting $C^\vee$, the $\varpi$-adic integral obtained from the distribution of $N_\infty\left( W^{(z)} \right) $ is nothing but a probability of survival. Proposition \ref{proposition:local2survival} is a consequence of three simple lemmas. 

\begin{lemma}[``Trivial key Lemma'']
\label{lemma:trivial_key_lemma}
 If $U$ and $U'$ are independent and distributed as $\Uc\left( \Oc \right)$, then for any two integers $a$ and $b$ in $\Z$:
 $$ \varpi^{a}U + \varpi^{b}U' \eqlaw \varpi^{\min\left(a,b\right)}U \ .$$
\end{lemma}
The lemma can be reformulated by saying that the Haar measure on $\Oc$ is tropicalization friendly. Although simple, this fact is crucial because the combinatorics of representation theory using crystals happen in the tropical world, where the operations $\left( \min, +, - \right)$ replace the algebraic $\left( + , \times, / \right)$ (see for example \cite{BFZ96}, or paragraph 42.2 in \cite{Lusztig93}). As an application, one observes the natural appearance of the minimal (signed) distances of the random walk $W^{(z)}$ to each wall of the (dual) Weyl chamber:
$$ \left( \inf_{0 \leq s \leq t} \langle \alpha, W^{(z)}_s \rangle \right)_{\left( \alpha \in \Delta \right)}$$
These are very reminiscent of the Littelmann path model, which we implicitly use. That matter is further discussed in the next subsection. 

\begin{lemma}
\label{lemma:exp_functionals}
For every $t \in \N \bigsqcup \left\{ \infty \right\}$, there exists a family of independent Haar distributed random variables $\left( C^\alpha \right)_{\left( \alpha \in \Delta \right)}$ in the ring of integers $\Oc$ such that:
\begin{align}
\label{eq:exp_functionals}
\forall \alpha \in \Delta, t \in \N, \ \chi_\alpha^-\left( N_{t}\left( W^{(z)} \right) \right)
                         & = \varpi^{ \inf_{0 \leq s \leq t} \langle \alpha, W^{(z)}_s \rangle } C^\alpha \ .
\end{align}
Note that the $\left( C^\alpha \right)_{\left( \alpha \in \Delta \right)}$ implicitly depend on $t$ and are not necessarily independent for different values of $t$.
\end{lemma}
\begin{proof}
Let us only consider $t<\infty$. The case $t=\infty$ is obtained by a limit argument. Because of equations (\ref{eq:def_N_t_1},~ \ref{eq:def_N_t_2}) and $\chi_\alpha^-$ being a character, we have that for every fixed simple root $\alpha$:
\begin{align*}
    \chi_\alpha^-\left( N_{t}\left( W^{(z)} \right) \right)
= & \chi_\alpha^-\left( n_0 \right)
    + \sum_{s=1}^t \chi_\alpha^-\left(  \varpi^{-W^{(z)}_{s-1}  } n_{s}' \varpi^{W^{(z)}_{s-1}}
                                        \varpi^{-W^{(z)}_{s  }}   n_{s}  \varpi^{W^{(z)}_{s}}   \right) \\
= & \chi_\alpha^-\left( \varpi^{-W^{(z)}_{t}}   n_{t}  \varpi^{W^{(z)}_{t}}   \right)
    + \sum_{s=0}^{t-1} \chi_\alpha^-\left(  \varpi^{-W^{(z)}_{s  }}   n_s n_{s+1}'  \varpi^{W^{(z)}_{s}}   \right) \\
= & \sum_{s=0}^t \varpi^{ \langle \alpha, W^{(z)}_s \rangle} \chi_\alpha^-\left( c_s \right) 
\end{align*}
where $\left( c_s ; 0 \leq s \leq t \right)$ are the independent and Haar distributed random variables on $N\left( \Oc \right)$, which we obtain from $c_s = n_s n_{s+1}'$ for $s<t$ and $c_t = n_t$. We now claim that the random variables $\left( \chi_\alpha^-\left( c_s \right) \right)_{ \left( 0 \leq s \leq t, \alpha \in \Delta \right)}$ are all independent and Haar distributed on $\Oc$. Independence for different $0 \leq s \leq t$ is clear. And for fixed $s$, we have by Haar invariance in $N\left( \Oc \right)$ that for every $u \in \Oc$ and simple roots $\left( \alpha, \beta \right) \in \Delta^2$:
$$ \chi_\alpha^-\left( c_s \right) \eqlaw \chi_\alpha^-\left( c_s x_{-\beta}(u) \right)
                                   =      \chi_\alpha^-\left( c_s  \right) + u \mathds{1}_{\alpha=\beta} \ .
$$
Therefore the vector $\left( \chi_\alpha^-\left( c_s \right) \right)_{\alpha \in \Delta}$ is Haar distributed on $\Oc^{\Card \Delta}$, which is the same as saying that the $\chi_\alpha^-\left( c_s \right)$ for different $\alpha \in \Delta$ are independent and Haar distributed.

In the end, the result is obtained by applying the ``trivial key Lemma'' \ref{lemma:trivial_key_lemma}.
\end{proof}

Recall that $\psi: \left( \Kc, + \right) \rightarrow \left( \C^*, \times \right)$ is an additive character, trivial on $\Oc$ but non trivial on $\varpi^{-1} \Oc$. Here is another simple but important property that bridges Haar integration over $\Oc$ and taking indicator functions:
\begin{lemma}[Averaging Lemma]
\label{lemma:averaging}
If $U \eqlaw \Uc\left( \Oc \right)$, then:
$$ \forall x \in \Kc, \E\left[ \psi\left( x U \right) \right] = \mathds{1}_{ \left\{ x \in \Oc \right\} } \ .$$
\end{lemma}
\begin{proof}
The claim is clearly true if $x \in \Oc$, as $\psi$ is identically $1$ on $\Oc$. If $x \notin \Oc$, there is an element $y \in \Oc$ such that $xy = \varpi^{-1}$. Haar invariance gives $U \eqlaw U + y$ and therefore $ x U \stackrel{\Lc}{=} x U + \varpi^{-1}$. Hence:
$$ \psi\left( x U \right) \stackrel{\Lc}{=} \psi\left( x U \right) \zeta$$
where $\zeta = \psi\left( \varpi^{-1} \right)$ is a non-trivial $q$-th root of unity. One concludes by the classical averaging roots of unity:
$$ 0 = \frac{1}{q} \sum_{k=0}^{q-1} \zeta^{k} = \frac{1}{q} \left( \sum_{k=0}^{q-1} \zeta^{k} \right)\E\left( \psi\left( x U \right) \right) 
     = \frac{1}{q} \sum_{k=0}^{q-1} \E\left( \zeta^k \psi\left( x U \right) \right)
     = \E\left[ \psi\left( x U \right) \right] \ .$$
%\begin{align*}
%0 = & \frac{1}{q} \sum_{k=0}^{q-1} \zeta^{k} \\
%  = & \frac{1}{q} \left( \sum_{k=0}^{q-1} \zeta^{k} \right)\E\left( \psi\left( x U \right) \right)\\
%  = & \frac{1}{q} \sum_{k=0}^{q-1} \E\left( \zeta^k \psi\left( x U \right) \right)\\
%  = & \E\left( \psi\left( x U \right) \right)
%\end{align*}
\end{proof}

\begin{proof}[ {\bf Proof of Proposition \ref{proposition:local2survival}} ]
The following sequence of equalities hold.
\begin{align*}
  \quad & \E\left[ \varphi_N\left( \varpi^{-\lambda^\vee} N_\infty\left( W^{(z)} \right) \varpi^{\lambda^\vee} \right) \right] \\
\stackrel{ \textrm{Eq.} \eqref{eq:def_psi_N} }{=}
        & \E\left[ \prod_{\alpha} \psi \circ \chi_\alpha^-\left( \varpi^{-\lambda^\vee} N_{\infty}\left( W^{(z)} \right) \varpi^{\lambda^\vee} \right) \right]\\
\stackrel{ \textrm{Lemma } \ref{lemma:exp_functionals} }{=}
        & \E\left[ \prod_{\alpha} \psi\left( \varpi^{ \langle \alpha, \lambda^\vee \rangle + \inf_{0 \leq s} \langle \alpha, W^{(z)}_s \rangle } C^\alpha \right) \right]\\
\stackrel{ \textrm{Lemma } \ref{lemma:averaging} }{=}
        & \E\left[ \prod_{\alpha} \mathds{1}_{\left\{ \inf_{0 \leq s} \langle \alpha, \lambda^\vee + W^{(z)}_s \rangle \geq 0 \right\} } \right] 
          \textrm{ (Conditionally on $W^{(z)}$)}\\
= \quad & \P\left( \forall \alpha \in \Delta,~ \inf_{0 \leq s} \langle \alpha, \lambda^\vee + W^{(z)}_s \rangle \geq 0 \right)\\
= \quad & \P\left( \lambda^\vee + W^{(z)} \textrm{ remains in } C^\vee \right) \ .
\end{align*}
\end{proof}

\begin{example}[Example \ref{example:PGL2} continued]
Let us illustrate the previous computations in the case of $PGL_2\left( \Kc \right)$. Recall that for $t \in \N \bigsqcup \{ \infty \}$:
$$ N_{t}\left( W^{(z)} \right) = \begin{pmatrix} 1 & 0\\ \sum_{s=0}^t c_s \varpi^{W_s^{(z)}} & 1 \end{pmatrix}$$
where the $\left(c_t ; t \in \N \right)$ are independent Haar distributed random variables on $\Oc$. The only character of $N$ gives the coefficient below the diagonal, and we have the equality in law:
$$ \sum_{s=0}^t c_s \varpi^{W_s^{(z)}} \eqlaw \varpi^{ \inf_{ 0 \leq s \leq t} W_s^{(z)}} \Uc\left( \Oc \right) \ .$$
Hence, if $\lambda^\vee = \lambda \omega^\vee_1$ with $\lambda \in \Z$:
\begin{align*}
  & \E\left[ \varphi_N\left( \varpi^{-\lambda^\vee} N_\infty\left( W^{(z)} \right) \varpi^{\lambda^\vee} \right) \right] \\
= & \E\left[ \varphi_N\left( \begin{pmatrix} \varpi^{ -\lambda} & 0\\ 0 & 1 \end{pmatrix}
                             \begin{pmatrix} 1 & 0\\ \varpi^{ \inf_{ 0 \leq s < \infty} W_s^{(z)}} \Uc\left( \Oc \right) & 1 \end{pmatrix}
                             \begin{pmatrix} \varpi^{\lambda} & 0\\ 0 & 1 \end{pmatrix}
                       \right)  \right] \\
= & \E\left[ \psi\left( \varpi^{ \lambda + \inf_{ 0 \leq s < \infty} W_s^{(z)} }
                        \Uc\left( \Oc \right)
                 \right)  \right]\\
= & \P\left( \lambda + \inf_{ 0 \leq s < \infty} W_s^{(z)} \geq 0 \right) \ .
\end{align*}
This is indeed the probability of the random walk $\lambda + W^{(z)}$ staying positive. For a simple random walk drifting to $\infty$, it is well-known that minus the infinimum is a geometric random variable. Here the parameter is $e^{-2z} < 1$. If $\lambda \geq 0$:
$$ \E\left[ \varphi_N\left( \varpi^{-\lambda^\vee} N_\infty\left( W^{(z)} \right) \varpi^{\lambda^\vee} \right) \right] 
 = \P\left( \textrm{Geom}\left( e^{-2z} \right) \leq \lambda \right)
 = 1-e^{-2z \left(\lambda + 1 \right) } \ .
 $$
\end{example}

In section \ref{section:reflection_principle}, we will prove a reflection principle that allows to relate the probability of $\lambda^\vee + W^{(z)}$ never exiting the Weyl chamber to the character of the irreducible representation $V\left( \lambda^\vee \right)$. This is exactly corollary 7.7 in \cite{LLP}, which is proved using the Littelmann path model, also in the minuscule case. We give an elementary proof using the reflection principle. 

\begin{proposition}
\label{proposition:probabilistic_char} 
If $\lambda^\vee$ is dominant:
$$ \ch V\left( \lambda^\vee \right)(z)
= b(z) e^{ \langle z, \lambda^\vee + \rho^\vee \rangle}
  \P\left( \lambda^\vee + W^{(z)}\textrm{ remains in } C^\vee \right) \ .$$
\end{proposition}
\begin{proof}
Consequence of Theorem \ref{thm:reflection_principle} and the Weyl character formula ( p.319 Theorem 5.75 in \cite{knapp02}).
\end{proof}

We will now discuss the relationship of the aforementioned results to crystal combinatorics and the similarities with the Archimedean case.

\subsection{Relation to crystals and similarities with the Archimedean case}
We used very indirectly crystal combinatorics, or more precisely the Littelmann path model, thanks to which each instance of the random walk $W^{(z)}$ can be viewed as a crystal element. More precisely, at every finite time $T \in \N$ and under the minuscule hypothesis, the discrete random walk $\left( W^{(z)}_t; t =0, 1, \dots, T \right)$ with increments in $W \Lambda^\vee$ is a crystal element which is obtained by successive tensor products of elements from the crystal with highest weight $\Lambda^\vee$. For further explanations, we refer to \cite{LLP}. In fact, Proposition \ref{proposition:probabilistic_char} holds in the context of general Kac-Moody groups as proved in the subsequent paper \cite{LLP2} of Lesigne {\it et al.}, but one has to observe Littelmann paths continuously in time. Thanks to the Satake correspondence, we have replaced tensor products by convolution products.

In the Archimedean case, virtually the same Poisson formula as in Proposition \ref{proposition:poisson_formula} holds with a Euclidian Brownian motion $\left( W_t^{(z)}; t \in \R_+ \right)$ instead of the random walk $\left( W^{(z)}_t; t =0, 1, 2, \dots\right)$ and a hypoelliptic Brownian motion $\left( B_t\left( W^{(z)} \right) ; t \in \R_+ \right)$ on the group instead of the random walk $\left( B_t\left( W^{(z)} \right) ; t \in \N \right)$. See subsection 6.2.3 of \cite{thesis}. Notice that the discrete setting is richer as it depends on the dominant coweight $\Lambda^\vee$, while the only canonical random walk on a continuous space is the Brownian motion. Also, there exists a related Littelmann path model for geometric crystal \cite{chh14a}. In that context, the Brownian motion $\left( W_t^{(z)}; t \in \R_+ \right)$ can be seen as \emph{geometric} Littelmann path.

In the same direction, the formula in Proposition \ref{proposition:probabilistic_char} has an analogue for the normalised Archimedean function (Proposition 5.4.1, (iii) in \cite{thesis}): instead of considering the probability of survival of a random walk killed upon exiting the Weyl chamber, one has to consider the probability of survival of a Brownian motion with a killing measure given by the Toda potential.

% --------------------------------------------------------------------
\section{The reflection principle for minuscule walks}
\label{section:reflection_principle}

It is understood that minuscule walks are basically the only reasonable random walks for which the reflection principle applies (see \cite{Biane} end of \textsection 2 ). The fact that increments lie in $W \Lambda^\vee$, the Weyl group orbit of $\Lambda^\vee$, insures that the walk $W^{(z)}$ is \emph{ reflectable }, i.e that $W^{(z)}$ cannot exit $C^\vee$ without having occupied a lattice point on a wall. For a classification of reflectable walks on general lattices, see \cite{GrabinerMagyar}. Their result boils down to taking minuscule steps or compatible ones. 

The following generalized reflection principle has been rediscovered on many occasions. In enumerative combinatorics, if the goal is to count the number of non-intersecting paths of a given length, the formula is known as the Lindstr\"om-Gessel-Viennot principle (see discussion in the section 2.7 of \cite{Stanley}). When dealing with continuous time Markov processes that are evolving independently, and with the goal of computing a non-intersection probability, the result is even older and dates back to Karlin and McGregor \cite{karlin59}. 

Here, we want to consider discrete time random walks with non-uniform weights and an infinite time horizon: the enumerative combinatorics approach would require using weights and a limiting argument from finite paths to infinite paths. Also, if the ambient lattice is isomorphic to $\Z^r$, most of the random walks at hand cannot be treated as $r$ independent walkers: the Karlin-MacGregor formula cannot be invoked as is. Here, we feel that applying existing theorems to our 
setting is tedious. Our result is more elegantly obtained by a direct proof, while of course using the same ideas. Figure \ref{fig:reflection_principle} serves as an illustration.

\begin{theorem}[Reflection principle]
\label{thm:reflection_principle}
We have:
$$ \P\left( \lambda^\vee + W^{(z)} \textrm{ remains in } C^\vee \right) 
 = \sum_{w \in W} \left(-1 \right)^{\ell(w)} e^{ \langle w \left( \lambda^\vee + \rho^\vee \right) - \left( \lambda^\vee + \rho^\vee \right), z \rangle  }$$
\end{theorem}

\addtocounter{equation}{1}
\setcounter{figure}{ \value{equation} }
\begin{figure}[ht!]
%\centering
\caption{Illustration of the reflection principle for the $A_2$-type weight lattice}
\label{fig:reflection_principle}
\begin{tikzpicture}[auto, bend right, scale=0.3]
\usetikzlibrary{positioning}
\node at(9.8119,16.9952){$H_{\alpha_1}$};
\node at(-4.7999,-8.3139){$H_{\alpha_1}$};
\path [draw, very thick, color = black] (9.8119,16.9952) --
(-4.7999,-8.3139);
\node at(-9.6,0.0){$H_{\alpha_2}$};
\node at(32.0,-0.0){$H_{\alpha_2}$};
\path [draw, very thick, color = black] (-9.6,0.0) -- (32.0,-0.0);
\node at(-4.7999,8.3139){$H_{\alpha_1 + \alpha_2}$};
\node at(9.5977,-16.6242){$H_{\alpha_1 + \alpha_2}$};
\path [draw, very thick, color = black] (-4.7999,8.3139) --
(9.5977,-16.6242);
\definecolor{color_walk}{rgb}{1.0,0.2,0.2}
\definecolor{color_reflected_walk}{rgb}{0,0,0.9}
\definecolor{color_merged_walk}{rgb}{0.9,0,0.9}
\path [draw, thick, color = color_walk] (7.5,4.33) -- (8.5,4.33) --
(9.5,4.33) -- (9.0,3.464) -- (10.0,3.464) -- (11.0,3.464) -- (10.5,4.33)
-- (10.0,5.196) -- (11.0,5.196) -- (12.0,5.196) -- (11.5,6.062) --
(12.5,6.062) -- (12.0,6.928) -- (11.5,7.794) -- (11.0,8.66) --
(12.0,8.66) -- (11.5,7.794) -- (11.0,8.66) -- (10.5,9.526) --
(11.5,9.526) -- (12.5,9.526) -- (12.0,8.66) -- (13.0,8.66) --
(12.5,9.526) -- (13.5,9.526) -- (13.0,8.66) -- (14.0,8.66) --
(15.0,8.66) -- (14.5,9.526) -- (15.5,9.526) -- (15.0,10.392) --
(16.0,10.392) -- (15.5,11.258) -- (16.5,11.258) -- (16.0,10.392) --
(15.5,9.526) -- (15.0,8.66) -- (14.5,9.526) -- (14.0,10.392) --
(13.5,11.258) -- (14.5,11.258) -- (14.0,10.392) -- (13.5,11.258) --
(14.5,11.258) -- (15.5,11.258) -- (16.5,11.258) -- (16.0,10.392) --
(17.0,10.392) -- (18.0,10.392) -- (19.0,10.392) -- (18.5,9.526) --
(19.5,9.526) -- (19.0,10.392) -- (18.5,9.526) -- (19.5,9.526) --
(19.0,10.392) -- (18.5,9.526) -- (18.0,10.392) -- (17.5,9.526) --
(18.5,9.526) -- (19.5,9.526) -- (19.0,10.392) -- (18.5,11.258) --
(19.5,11.258) -- (19.0,10.392) -- (18.5,9.526) -- (18.0,10.392) --
(17.5,9.526) -- (18.5,9.526) -- (19.5,9.526) -- (19.0,8.66) --
(18.5,9.526) -- (18.0,10.392) -- (19.0,10.392) -- (20.0,10.392) --
(19.5,11.258) -- (20.5,11.258) -- (20.0,12.124) -- (19.5,11.258) --
(20.5,11.258) -- (20.0,12.124) -- (21.0,12.124) -- (20.5,12.99) --
(20.0,12.124) -- (19.5,11.258) -- (19.0,10.392) -- (20.0,10.392) --
(19.5,11.258) -- (20.5,11.258) -- (20.0,10.392) -- (19.5,11.258) --
(19.0,12.124) -- (20.0,12.124) -- (21.0,12.124) -- (22.0,12.124) --
(21.5,11.258) -- (21.0,10.392) -- (22.0,10.392) -- (21.5,9.526) --
(21.0,10.392) -- (20.5,9.526) -- (20.0,10.392) -- (19.5,11.258) --
(19.0,10.392) -- (18.5,9.526) -- (18.0,8.66) -- (19.0,8.66) --
(20.0,8.66) -- (19.5,9.526) -- (19.0,8.66) -- (18.5,9.526) --
(18.0,8.66) -- (17.5,7.794) -- (17.0,6.928) -- (16.5,7.794) --
(16.0,6.928) -- (17.0,6.928) -- (16.5,6.062) -- (17.5,6.062) --
(17.0,5.196) -- (18.0,5.196) -- (17.5,6.062) -- (17.0,6.928) --
(18.0,6.928) -- (19.0,6.928) -- (18.5,7.794) -- (18.0,6.928) --
(17.5,6.062) -- (17.0,6.928) -- (18.0,6.928) -- (17.5,7.794) --
(18.5,7.794) -- (19.5,7.794) -- (20.5,7.794) -- (21.5,7.794) --
(21.0,6.928) -- (20.5,7.794) -- (20.0,6.928) -- (19.5,6.062) --
(19.0,6.928) -- (20.0,6.928) -- (19.5,7.794) -- (19.0,6.928) --
(18.5,7.794) -- (19.5,7.794) -- (20.5,7.794) -- (20.0,8.66) --
(21.0,8.66) -- (22.0,8.66) -- (21.5,7.794) -- (22.5,7.794) --
(22.0,6.928) -- (21.5,7.794) -- (21.0,8.66) -- (20.5,7.794) --
(21.5,7.794) -- (21.0,6.928) -- (22.0,6.928) -- (23.0,6.928) --
(22.5,7.794) -- (23.5,7.794) -- (23.0,8.66) -- (22.5,7.794) --
(23.5,7.794) -- (23.0,8.66) -- (24.0,8.66) -- (23.5,9.526) --
(23.0,8.66) -- (22.5,9.526) -- (22.0,10.392) -- (23.0,10.392) --
(22.5,9.526) -- (22.0,8.66) -- (23.0,8.66) -- (24.0,8.66) --
(23.5,7.794) -- (23.0,6.928) -- (22.5,6.062) -- (22.0,5.196) --
(21.5,4.33) -- (21.0,3.464) -- (20.5,4.33) -- (21.5,4.33) --
(21.0,3.464) -- (22.0,3.464) -- (21.5,4.33) -- (21.0,3.464) --
(22.0,3.464) -- (21.5,4.33) -- (22.5,4.33) -- (23.5,4.33) --
(23.0,3.464) -- (24.0,3.464) -- (25.0,3.464) -- (24.5,2.598) --
(24.0,3.464) -- (25.0,3.464) -- (26.0,3.464) -- (25.5,2.598) --
(26.5,2.598) -- (26.0,1.732) -- (25.5,2.598) -- (25.0,3.464) --
(24.5,2.598) -- (25.5,2.598) -- (25.0,1.732) -- (24.5,0.866) --
(24.0,0.0);
\path [draw, thick, color = color_reflected_walk] (7.5,-4.33) --
(8.5,-4.33) -- (9.5,-4.33) -- (9.0,-3.464) -- (10.0,-3.464) --
(11.0,-3.464) -- (10.5,-4.33) -- (10.0,-5.196) -- (11.0,-5.196) --
(12.0,-5.196) -- (11.5,-6.062) -- (12.5,-6.062) -- (12.0,-6.928) --
(11.5,-7.794) -- (11.0,-8.66) -- (12.0,-8.66) -- (11.5,-7.794) --
(11.0,-8.66) -- (10.5,-9.526) -- (11.5,-9.526) -- (12.5,-9.526) --
(12.0,-8.66) -- (13.0,-8.66) -- (12.5,-9.526) -- (13.5,-9.526) --
(13.0,-8.66) -- (14.0,-8.66) -- (15.0,-8.66) -- (14.5,-9.526) --
(15.5,-9.526) -- (15.0,-10.392) -- (16.0,-10.392) -- (15.5,-11.258) --
(16.5,-11.258) -- (16.0,-10.392) -- (15.5,-9.526) -- (15.0,-8.66) --
(14.5,-9.526) -- (14.0,-10.392) -- (13.5,-11.258) -- (14.5,-11.258) --
(14.0,-10.392) -- (13.5,-11.258) -- (14.5,-11.258) -- (15.5,-11.258) --
(16.5,-11.258) -- (16.0,-10.392) -- (17.0,-10.392) -- (18.0,-10.392) --
(19.0,-10.392) -- (18.5,-9.526) -- (19.5,-9.526) -- (19.0,-10.392) --
(18.5,-9.526) -- (19.5,-9.526) -- (19.0,-10.392) -- (18.5,-9.526) --
(18.0,-10.392) -- (17.5,-9.526) -- (18.5,-9.526) -- (19.5,-9.526) --
(19.0,-10.392) -- (18.5,-11.258) -- (19.5,-11.258) -- (19.0,-10.392) --
(18.5,-9.526) -- (18.0,-10.392) -- (17.5,-9.526) -- (18.5,-9.526) --
(19.5,-9.526) -- (19.0,-8.66) -- (18.5,-9.526) -- (18.0,-10.392) --
(19.0,-10.392) -- (20.0,-10.392) -- (19.5,-11.258) -- (20.5,-11.258) --
(20.0,-12.124) -- (19.5,-11.258) -- (20.5,-11.258) -- (20.0,-12.124) --
(21.0,-12.124) -- (20.5,-12.99) -- (20.0,-12.124) -- (19.5,-11.258) --
(19.0,-10.392) -- (20.0,-10.392) -- (19.5,-11.258) -- (20.5,-11.258) --
(20.0,-10.392) -- (19.5,-11.258) -- (19.0,-12.124) -- (20.0,-12.124) --
(21.0,-12.124) -- (22.0,-12.124) -- (21.5,-11.258) -- (21.0,-10.392) --
(22.0,-10.392) -- (21.5,-9.526) -- (21.0,-10.392) -- (20.5,-9.526) --
(20.0,-10.392) -- (19.5,-11.258) -- (19.0,-10.392) -- (18.5,-9.526) --
(18.0,-8.66) -- (19.0,-8.66) -- (20.0,-8.66) -- (19.5,-9.526) --
(19.0,-8.66) -- (18.5,-9.526) -- (18.0,-8.66) -- (17.5,-7.794) --
(17.0,-6.928) -- (16.5,-7.794) -- (16.0,-6.928) -- (17.0,-6.928) --
(16.5,-6.062) -- (17.5,-6.062) -- (17.0,-5.196) -- (18.0,-5.196) --
(17.5,-6.062) -- (17.0,-6.928) -- (18.0,-6.928) -- (19.0,-6.928) --
(18.5,-7.794) -- (18.0,-6.928) -- (17.5,-6.062) -- (17.0,-6.928) --
(18.0,-6.928) -- (17.5,-7.794) -- (18.5,-7.794) -- (19.5,-7.794) --
(20.5,-7.794) -- (21.5,-7.794) -- (21.0,-6.928) -- (20.5,-7.794) --
(20.0,-6.928) -- (19.5,-6.062) -- (19.0,-6.928) -- (20.0,-6.928) --
(19.5,-7.794) -- (19.0,-6.928) -- (18.5,-7.794) -- (19.5,-7.794) --
(20.5,-7.794) -- (20.0,-8.66) -- (21.0,-8.66) -- (22.0,-8.66) --
(21.5,-7.794) -- (22.5,-7.794) -- (22.0,-6.928) -- (21.5,-7.794) --
(21.0,-8.66) -- (20.5,-7.794) -- (21.5,-7.794) -- (21.0,-6.928) --
(22.0,-6.928) -- (23.0,-6.928) -- (22.5,-7.794) -- (23.5,-7.794) --
(23.0,-8.66) -- (22.5,-7.794) -- (23.5,-7.794) -- (23.0,-8.66) --
(24.0,-8.66) -- (23.5,-9.526) -- (23.0,-8.66) -- (22.5,-9.526) --
(22.0,-10.392) -- (23.0,-10.392) -- (22.5,-9.526) -- (22.0,-8.66) --
(23.0,-8.66) -- (24.0,-8.66) -- (23.5,-7.794) -- (23.0,-6.928) --
(22.5,-6.062) -- (22.0,-5.196) -- (21.5,-4.33) -- (21.0,-3.464) --
(20.5,-4.33) -- (21.5,-4.33) -- (21.0,-3.464) -- (22.0,-3.464) --
(21.5,-4.33) -- (21.0,-3.464) -- (22.0,-3.464) -- (21.5,-4.33) --
(22.5,-4.33) -- (23.5,-4.33) -- (23.0,-3.464) -- (24.0,-3.464) --
(25.0,-3.464) -- (24.5,-2.598) -- (24.0,-3.464) -- (25.0,-3.464) --
(26.0,-3.464) -- (25.5,-2.598) -- (26.5,-2.598) -- (26.0,-1.732) --
(25.5,-2.598) -- (25.0,-3.464) -- (24.5,-2.598) -- (25.5,-2.598) --
(25.0,-1.732) -- (24.5,-0.866) -- (24.0,0.0);
\path [draw, thick, dashed, color = color_merged_walk, ->] (24.5,-0.866) --
(24.0,0.0) -- (23.5,0.866) -- (24.5,0.866) -- (24.0,0.0) --
(23.5,-0.866) -- (23.0,-1.732) -- (24.0,-1.732) -- (25.0,-1.732) --
(24.5,-2.598) -- (24.0,-3.464) -- (23.5,-2.598) -- (23.0,-3.464) --
(22.5,-4.33) -- (22.0,-5.196) -- (23.0,-5.196) -- (24.0,-5.196) --
(25.0,-5.196) -- (24.5,-4.33) -- (24.0,-3.464) -- (23.5,-4.33) --
(23.0,-3.464) -- (22.5,-2.598) -- (22.0,-1.732) -- (23.0,-1.732) --
(22.5,-0.866) -- (23.5,-0.866) -- (23.0,0.0) -- (22.5,0.866) --
(22.0,1.732) -- (21.5,2.598) -- (21.0,1.732) -- (22.0,1.732) --
(21.5,0.866) -- (22.5,0.866) -- (22.0,1.732) -- (23.0,1.732) --
(22.5,0.866) -- (23.5,0.866) -- (24.5,0.866) -- (25.5,0.866) --
(25.0,0.0) -- (24.5,-0.866) -- (24.0,-1.732) -- (25.0,-1.732) --
(24.5,-0.866) -- (24.0,0.0) -- (23.5,-0.866) -- (23.0,0.0) -- (24.0,0.0)
-- (25.0,0.0) -- (24.5,-0.866) -- (24.0,0.0) -- (23.5,0.866) --
(23.0,0.0) -- (22.5,0.866) -- (22.0,0.0) -- (21.5,0.866) -- (21.0,1.732)
-- (22.0,1.732) -- (23.0,1.732) -- (22.5,0.866) -- (23.5,0.866) --
(23.0,1.732) -- (24.0,1.732) -- (23.5,0.866) -- (24.5,0.866) --
(24.0,0.0) -- (25.0,0.0) -- (26.0,0.0) -- (27.0,0.0) -- (28.0,0.0) --
(29.0,0.0) -- (30.0,0.0) -- (31.0,0.0) -- (32.0,0.0) -- (31.5,0.866) --
(31.0,1.732) -- (30.5,0.866) -- (30.0,0.0) -- (29.5,-0.866) --
(30.5,-0.866) -- (30.0,0.0) -- (31.0,0.0) -- (30.5,0.866) --
(30.0,1.732) -- (29.5,2.598) -- (30.5,2.598) -- (30.0,3.464) --
(31.0,3.464) -- (32.0,3.464);
% tau
\node at(24.0,0.0){$\times$};
\node at(24.0,-1.0){$\tau$};
% starting point
\node at (7.5,4.33){$\times$};
\node at (7.5,5.33){$\lambda^{\vee}$};
\node at (7.5,-4.33){$\times$};
\node at (7.5,-5.33){$s_\alpha \lambda^\vee$};

%Legend
\node (legend_anchor1) at(-9.6,-12.6242) {Rw before $\tau$};
\node (legend_anchor2) [below=0.1cm of legend_anchor1] {Reflected rw};
\node (legend_anchor3) [below=0.1cm of legend_anchor2] {Rw after $\tau$};
\node (A1) [right of=legend_anchor1] {}; \node (B1) [right of=A1] {};
\node (A2) [right of=legend_anchor2] {}; \node (B2) [right of=A2] {};
\node (A3) [right of=legend_anchor3] {}; \node (B3) [right of=A3] {};

\path [draw, thick, color = color_walk, ->]                   (A1) -- (B1); %(-9.6,-14.6242) -- (-7.6,-14.6242);
\path [draw, thick, color = color_reflected_walk, ->]         (A2) -- (B2); %(-9.6,-14.6242) -- (-7.6,-14.6242);
\path [draw, thick, dashed, color = color_merged_walk, ->]    (A3) -- (B3); %(-9.6,-14.6242) -- (-7.6,-14.6242);

%Steps
\node (steps_anchor) at(-9.6,12.6242) {};
\node (steps_anchor_text) [left of=steps_anchor] {Steps:};
\path [draw, thick, ->]                   (steps_anchor.center) -- ++ ( 1.000, 0.000);
\path [draw, thick, ->]                   (steps_anchor.center) -- ++ (-0.500, 0.866);
\path [draw, thick, ->]                   (steps_anchor.center) -- ++ (-0.500,-0.866);

\end{tikzpicture}
%Legend:
%\begin{itemize}
% \item The minuscule walk starts at $\lambda^\vee$ and has steps according to the weights of the geometric representation $V(\omega_1)$.
% \item In red, the walk before hitting a wall at $\tau$. In blue the reflected part. In violet, the unchanged part after $\tau$.
%\end{itemize}
\end{figure}
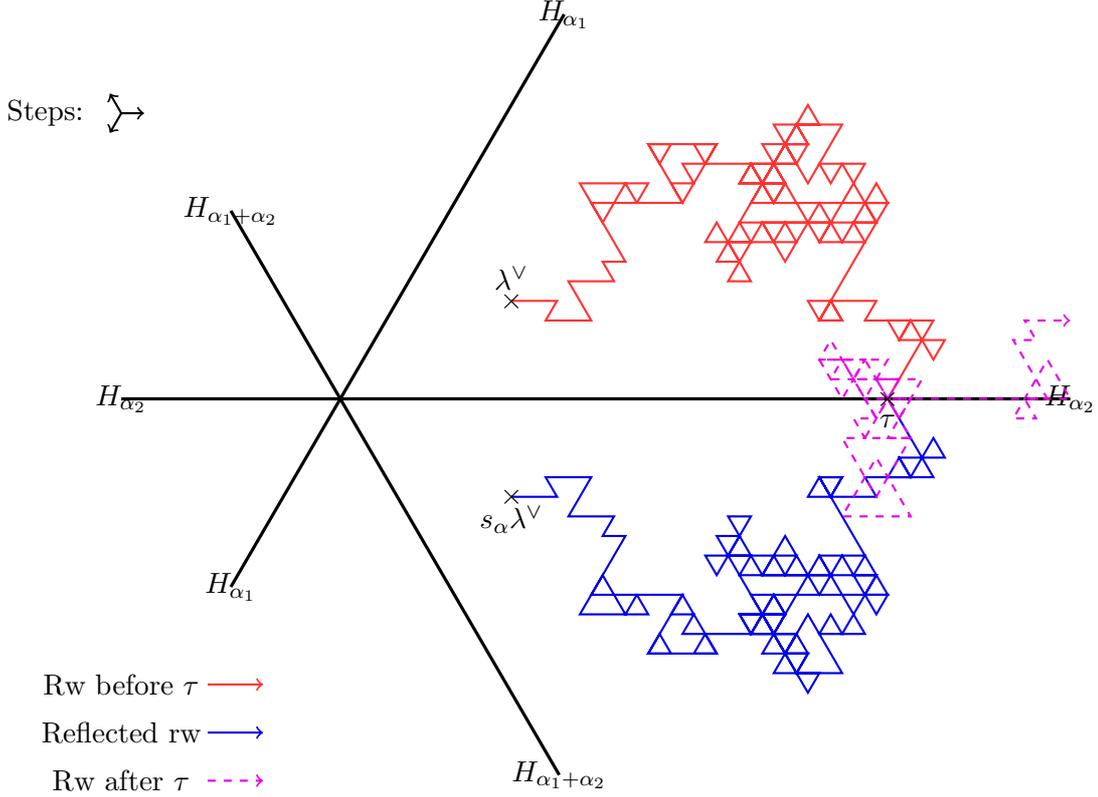

\begin{proof}
Recall that because of the law of large numbers, almost surely, $W^{(z)}$ will eventually enter the Weyl chamber and stay there. By performing a translation by the Weyl covector $\rho^\vee$, we see that $\lambda^\vee + W^{(z)}$ remains in the dual Weyl chamber (walls included) if and only if $\rho^\vee + \lambda^\vee + W^{(z)}$ never hits a wall. This shift is classical. Let
$$ \tau := \inf\left\{ t \in \N \ | \ \rho^\vee + \lambda^\vee + W_{t}^{(z)} \in \partial C^\vee \right\}$$
be the first time $\rho^\vee + \lambda^\vee + W_{t}^{(z)}$ hits a wall. The convention is that $\tau = \infty$ if the event never occurs.

The sample path space is $\Pi = \left( X^\vee \right)^{\N}$. Consider the functional $F: \Pi \rightarrow \R$ defined by:
$$ \forall \pi \in \Pi, \ F\left( \pi \right) := \sum_{w \in W} \left(-1 \right)^{\ell(w)} \mathds{1}_{ \left\{ \pi(0) \in w C^\vee \right\} } \ .$$

As a first ingredient for the reflection principle, we want to define an involutive transform $T: \Pi \rightarrow \Pi$ that reflects the portion of a path before $\tau$, if $\tau<\infty$. If $\tau = \infty$, we set $T\left( \pi \right) = \pi$. Hence:
\begin{align}
\label{eq:tau_infinite}
\tau = \infty & \Rightarrow F \circ T (\pi) = F(\pi) \ .
\end{align}
The reflection considered is with respect to the wall $H_\alpha = \Ker \alpha$ that the random walk hits first. However, if the exit point is on more than one wall, there is an ambiguity in our choice of wall. To that effect, one can choose an arbitrary order on the walls (or equivalently on positive roots). This choice is not essential as its only purpose is to give a well-defined map $T$. If the path $\pi$ hits the hyperplane $H_\alpha$ at the time $\tau$, $\alpha$ being the smallest in our arbitrary order, we set:
$$ T(\pi)_t := \left\{ \begin{array}{cc}
                   s_\alpha \left( \pi(t) \right) & \textrm{ if } t < \tau \ ,\\
                   \pi(t)                         & \textrm{ otherwise. }
               \end{array} \right.
$$
The transform is clearly involutive on the sample path space $\Pi$ and leaves $\tau$ unchanged. Moreover:
\begin{align}
\label{eq:tau_finite}
\tau < \infty & \Rightarrow F \circ T (\pi) = -F(\pi) \ .
\end{align}

At the moment, the transform $T$ does not preserve the distribution of $\rho^\vee + \lambda^\vee + W^{(z)}$. In order to compensate for that issue, we randomize the starting point. This is the second ingredient in the reflection principle, which is implicit in enumerative combinatorics where one counts a number of finite paths. Consider a random coweight $\mu^\vee$ distributed on the Weyl group orbit of $\lambda^\vee + \rho^\vee$ as follows:
\begin{align}
\label{eq:randomization}
\P\left( \mu^\vee = w \left( \lambda^\vee + \rho^\vee \right) \right)
& = \frac{                 e^{ \langle w \left( \lambda^\vee + \rho^\vee \right), z \rangle  }}
         { \sum_{w' \in W} e^{ \langle w'\left( \lambda^\vee + \rho^\vee \right), z \rangle  }} \ .
\end{align}

We claim that, upon this randomization, the distribution of $\mu^\vee + W^{(z)}$ is preserved under $T$. In order to see that, it suffices to check that for any finite time horizon $t \in \N$:
\begin{align}
\label{eq:required_T}
\left( T\left( \mu^\vee + W^{(z)} \right)_s ; 0 \leq s \leq t \right) & \eqlaw \left( \mu^\vee + W^{(z)}_s ; 0 \leq s \leq t \right) \ .
\end{align}
Notice that $T\left( \mu^\vee + W^{(z)} \right)$ and $\mu^\vee + W^{(z)}$ coincide after $\tau$, which is a stopping time. By the strong Markov property for $W^{(z)}$, the question is reduced to proving that for every path $\pi \in \Pi$, hitting $\partial C^\vee$ at time $t$ for the first time, we have:
$$ \P\left( \forall 0 \leq s \leq t, \ T\left( \mu^\vee + W^{(z)} \right)_s = \pi(s) \right) = \P\left( \forall 0 \leq s \leq t, \ \mu^\vee + W^{(z)}_s = \pi(s) \right) \ .$$
Now, the probability of these finite paths is the product, along each path, of the transition probabilities from equation \eqref{eq:rw_lattice}. We see that these probabilities depend only on the increment $\pi(t) - \pi(0)$. More precisely, if $\pi$ hits $H_\alpha = \Ker \alpha$ first, the former equation becomes:
$$ \P\left( \mu^\vee = s_\alpha \pi(0) \right) e^{ \langle s_\alpha (\pi(t) - \pi(0)), z \rangle  }
 = \P\left( \mu^\vee = \pi(0) \right) e^{ \langle \pi(t) - \pi(0), z \rangle  } \ .$$
Because $\pi(t) \in H_\alpha$, $s_\alpha \pi(t) = \pi(t)$. Hence, in the end, equation \eqref{eq:required_T} is equivalent to:
$$ \P\left( \mu^\vee = s_\alpha \pi(0) \right) e^{ \langle - s_\alpha \pi(0), z \rangle  }
 = \P\left( \mu^\vee = \pi(0) \right) e^{ \langle - \pi(0), z \rangle  } \ .$$
From the previous equation, one sees that our choice (equation \eqref{eq:randomization}) for the distribution of $\mu^\vee$ was in fact the only possible choice.

We are now ready to finish the proof. On the one hand, because of implication \eqref{eq:tau_infinite}, and the fact that $\mu^\vee + W^{(z)}$ has a chance of never hitting $\partial C^\vee$ only in the case of $\mu^\vee = \lambda^\vee + \rho^\vee$, we have:
$$ \E\left( F\left( \mu^\vee + W^{(z)} \right) \mathds{1}_{\{ \tau = \infty \}} \right) 
 = \P\left( \mu^\vee = \lambda^\vee + \rho^\vee \right)
   \P\left( \lambda^\vee + \rho^\vee + W^{(z)} \textrm{ never hits } \partial C^\vee \right) \ .$$
On the other hand, because of implication \eqref{eq:tau_finite} and the fact that $T$ preserves the distribution of $\rho^\vee + \Lambda^\vee + W^{(z)}$, we also have:
$$ \E\left( F\left( \mu^\vee + W^{(z)} \right) \mathds{1}_{\{ \tau < \infty \}} \right) = 0 \ .$$

Therefore, by putting together the two previous equations:
\begin{align*}
    \P\left( \lambda^\vee + W^{(z)} \textrm{ remains in } C^\vee \right)
= & \P\left( \lambda^\vee + \rho^\vee + W^{(z)} \textrm{ never hits } \partial C^\vee \right)\\
= & \frac{ \E\left( F\left( \mu^\vee + W^{(z)} \right) \mathds{1}_{\left\{ \tau = \infty \right\}} \right) }{ \P\left( \mu^\vee = \lambda^\vee + \rho^\vee \right) }\\
= & \frac{ \E\left( F\left( \mu^\vee + W^{(z)} \right) \right) }{ \P\left( \mu^\vee = \lambda^\vee + \rho^\vee \right) }\\
= & \sum_{w \in W} \left(-1 \right)^{\ell(w)} \frac{ \P\left( \mu^\vee \in w C^\vee \right) }{ \P\left( \mu^\vee \in C^\vee \right) }\\
= & \sum_{w \in W} \left(-1 \right)^{\ell(w)} e^{ \langle w \left( \lambda^\vee + \rho^\vee \right) - \left( \lambda^\vee + \rho^\vee \right), z \rangle  } \ .
\end{align*}
\end{proof}

% --------------------------------------------------------------------
\section{Harmonic properties of Whittaker functions}
\label{section:harmonicity_properties}

We will now review the facts we will need around the Satake isomorphism and spherical functions. For more details, the reader is referred to \cite{Gross} for a pedagogical overview of the Satake isomorphism and to \cite{Cartier} for a more complete survey, including $p$-adic representation theory. The complete reference for spherical functions is \cite{macdonald}.

\subsection{About the Satake isomorphism}
\label{subsection:satake}
Given two locally constant functions $\left( \varphi_1, \varphi_2 \right) \in \Fc\left( G \right)$, one can define their convolution $\varphi_1 *_G \varphi_2$ by:
$$ \forall g \in G, \varphi_1 *_G \varphi_2(g) := \int_G \varphi_1\left( g h^{-1} \right) \varphi_2\left( h \right) dh \ .$$
This is possible as soon as one the two functions is compactly supported. Let $H\left( G,K \right) = \Fc\left( K \backslash G / K \right)$ be the space of (compactly supported) $K$-bi-invariant functions on the group. It is also known as the spherical Hecke algebra. Since the convolution of two bi-invariant functions is bi-invariant, $H\left( G,K \right)$ is called a convolution algebra.

In the same fashion, let $H\left( T, T\left( \Oc \right) \right)$ denote the space of functions on $T$, invariant under $T\left( \Oc \right) = T \cap K$. It is also a convolution $\C$-algebra with a much simpler structure. It is spanned by the indicator functions $\left( \mathds{1}_{\left\{ \varpi^{\mu^\vee} T\left( \Oc \right) \right\}} \right)_{\mu^\vee \in X^\vee}$. When we identify $\mathds{1}_{\left\{ \varpi^{\mu^\vee} T\left( \Oc \right) \right\}}$ with the formal exponential $e^{\mu^\vee}$, we have a $\C$-algebra isomorphism $H\left( T, T\left( \Oc \right) \right) \approx \C\left[ X^\vee \right]$. For an unramified character $\chi: T \rightarrow \C^*$ and a function $f \in H\left( T, T\left( \Oc \right) \right)$, we write the convolution product:
\begin{align}
\label{eq:evaluation}
f\left( \chi \right) := \chi *_T f \left( \Id \right)
& = \int_T f\left( t^{-1} \right) \chi(t) dt \ .
\end{align}
Notice that, if the character $\chi$ is given by a $z \in X \otimes_\Z \C$ such that for all $\mu^\vee \in X^\vee$, $\chi\left( \varpi^{-\mu^\vee} \right) = e^{\langle z, \mu^\vee \rangle}$, we have $e^{\mu^\vee}\left( \chi \right) = e^{\langle z, \mu^\vee \rangle}$. Therefore, one can simply think of the notation in equation \eqref{eq:evaluation} as an evaluation.

The Satake transform $\Sc: H\left( G,K \right) \rightarrow H\left( T, T\left( \Oc \right) \right)$ is defined by an integral transform similar to those studied by Harish-Chandra, in the real case. For $\Hc \in H\left( G, K \right)$, $\Sc\left( \Hc \right)$ is given by:
\begin{align}
\label{eq:satake_def} 
\forall t \in T, 
   \Sc\left( \Hc \right)\left( t \right)
:= \delta(t)^{ \half} \int_N dn \ \Hc\left( t n \right)
 = \delta(t)^{-\half} \int_N dn \ \Hc\left( n t \right) \ .
\end{align}

In fact, the Satake transform takes its values in $\C\left[ X^\vee \right]^W$, the space of functions on $X^\vee$ invariant under $W$. An even more involved statement is:

\begin{theorem}[Satake \cite{Satake} - see also e.g \cite{Gross} Proposition 3.6 ]
The Satake transform is an algebra isomorphism between the spherical Hecke algebra $H(G,K)$ and $\C\left[ X^\vee \right]^W$.
\end{theorem}

The importance of the Satake transform for us stems from:
\begin{proposition}
\label{proposition:eigenfunction_property}
Both the spherical vector $\Phi_\chi$ and the Whittaker function $\Wc_\chi$ are convolution eigenfunctions by elements in $\Hc \in H\left(G,K\right)$:
\begin{align}
\label{eq:convolution_eigeneq_sph}
\Phi_\chi *_G \Hc & = \Sc(\Hc)\left( \chi \right) \Phi_\chi \ ,\\
\label{eq:convolution_eigeneq}
\Wc_\chi  *_G \Hc & = \Sc(\Hc)\left( \chi \right) \Wc_\chi  \ .
\end{align}
\end{proposition}
%\begin{rmk}
%In fact, this characterises the local Whittaker function thanks to the multiplicity one theorem \cite{}. 

%However, this theorem fails in the metaplectic case (see McNamara \cite{McNamara}). This theorem can serve as a characterisation. In the case of local field of functions on a curve, an analogue exists in (Lafforgue \cite{Lafforgue}). It would be nice to find a reference for the most general case. Moreover, construct a characterisation in the metaplectic case that uses convolution.
%\end{rmk}
\begin{proof}
Here, we will use the Iwasawa decomposition $G = KAN$ (instead of $NAK$) along with the fact that the Haar measure decomposes as (see e.g \cite{Cartier} \textsection 4.1):
$$ dg = \delta(a) \ dk \ da \ dn \ .$$
We only need to prove both identities on $NA$ since the convolution of two right $K$ invariant functions is right $K$ invariant. Moreover, the terms on both side of an equality have the same behavior. In the first identity, they are left $N$ invariant. In the second identity, they have the twisting property of equation \eqref{eq:torsion} under the left action of $N$. This reduces the reasoning to proving that for $a \in A$:
\begin{align*}
\Phi_\chi *_G \Hc(a) 
 = \quad & \int_G \Phi_\chi(a h^{-1}) \Hc(h) dh\\
\stackrel{ \textrm{Iwasawa} }{=}
         & \int_A db \ \delta(b) \int_N dn \ \Phi_\chi\left(a \left(n b\right)^{-1} \right) \Hc(bn)\\
 = \quad & \int_A db \ \delta(b) \ \Phi_\chi(a b^{-1} ) \int_N dn \ \Hc(bn)\\
 = \quad & \Phi_\chi(a) \int_A db \ \delta(b) \Phi_\chi(b^{-1}) \int_N dn \ \Hc(bn)\\
 = \quad & \Phi_\chi(a) \int_A db \ \chi(b)^{-1} \delta\left( b\right)^{\half} \int_N dn \ \Hc(bn)\\
 = \quad & \Phi_\chi(a) \int_A db \ \chi(b)^{-1} \Sc\left( \Hc \right)(b)\\
 = \quad & \Sc(\Hc)\left( \chi \right) \Phi_\chi(a) \ .
\end{align*}
And:
\begin{align*}
\Wc_\chi *_G \Hc(a)
= \quad & \int_G \Wc_\chi(a h^{-1}) \Hc(h) dh \\
\stackrel{ \textrm{Eq } \eqref{eq:whittaker_def} }{=}
  & \int_N dn \int_G dh \Phi_{ \chi^{w_0} }( \bar{w}_0 n a h^{-1}) \Hc(h) \varphi_N(n)^{-1} \\
= \quad & \int_N dn \ \Phi_{ \chi^{w_0} } *_G \Hc( n a ) \varphi_N(n)^{-1} \\
\stackrel{ \textrm{Eq } \eqref{eq:convolution_eigeneq_sph} }{=}
  & \Sc(\Hc)\left( \chi^{w_0} \right) \int_N dn \ \Phi_{ \chi^{w_0} }( \bar{w}_0 n a ) \varphi_N(n)^{-1}\\
= \quad & \Sc(\Hc)\left( \chi \right) \Wc_\chi(a) \ .
\end{align*}
Interchanging integrals in the previous computations is allowed because the function $\Hc$ is compactly supported and the integral defining $\Wc_\chi$ is absolutely convergent.
\end{proof}

For a probabilist, being a convolution eigenfunction is very close to being harmonic for a certain random walk on $G$. To that endeavor, the eigenvalue would need to be one and the spherical Hecke algebra element must give rise to a probability measure. Therefore, the previous theorem is at the heart of our approach, once restated in a probabilistic fashion. 

In order to carry out our program, we will need the explicit computation of $\Sc\left( \Hc \right)$ for certain $\Hc$.

\subsection{On Macdonald's spherical function}
\label{subsection:macdonald}

Because of the Cartan decomposition (Eq. \eqref{eq:cartan_decomposition}), the indicator functions $\mathds{1}_{\left\{ K \varpi^{\lambda^\vee} K\right\}}$ of double $K$-cosets form a basis of $H\left( G, K \right)$. Therefore, the symmetric functions $\Sc\left( \lambda^\vee \right) := \Sc\left( \mathds{1}_{\left\{ K \varpi^{\lambda^\vee} K\right\}} \right)$ form a basis of $\C\left[ X^\vee \right]^W$. They coincide in fact with the Macdonald spherical functions (\cite{macdonald} Proposition 3.3.1). We record in the next proposition two well-known facts. The first one is that $\Sc\left( \lambda^\vee \right)$ encodes the cardinalities of $G/K$ cosets when intersecting strata from the Iwasawa and Cartan decompositions. The second one is that $\Sc\left( \lambda^\vee \right)$ is proportional to a character when $\lambda^\vee$ is minuscule.
 
If $A \subset G$ is a compact left $K$-invariant subset, then $A$ is a finite union of $G / K$ cosets. The number of such cosets is denoted by:
$$ \Card_{G/K}\left( A \right) = \int_A dg \ .$$

\begin{proposition}[Expressions for the spherical functions]
\label{proposition:value_S}
For every $\lambda^\vee$ dominant:
\begin{align}
\label{eq:spherical_function}
\Sc\left( \lambda^\vee \right) & = \sum_{\mu^\vee \in X^\vee} \Card_{G/K}\left( N \varpi^{-\mu^\vee} K \cap K \varpi^{-\lambda^\vee} K \right) q^{-\langle \mu^\vee, \rho \rangle} e^{\mu^\vee} \ ,
\end{align}
and if $\lambda^\vee$ is minuscule:
\begin{align}
\label{eq:spherical_function_minuscule}
\Sc\left( \lambda^\vee \right) = q^{\langle \rho, \lambda^\vee \rangle} \ch V\left( \lambda^\vee \right) \ .
\end{align}
\end{proposition}
\begin{proof}
We will use the fact that, in the Iwasawa decomposition $G=NAK$, the Haar measure decomposes as (see again e.g \cite{Cartier} \textsection 4.1):
$$ dg = \delta^{-1}(a) \ dk \ da \ dn$$
In the identification of the indicator function $\mathds{1}_{\left\{ \varpi^{\mu^\vee} T\left( \Oc \right) \right\}}$ with the formal exponential $e^{\mu^\vee}$ for $\mu^\vee \in X^\vee$, we have:
\begin{align*}
\Sc\left( \lambda^\vee \right)
= \quad & \sum_{\mu^\vee \in X^\vee}\Sc\left( \mathds{1}_{\left\{ K \varpi^{\lambda^\vee} K\right\}}\right)\left( \varpi^{\mu^\vee} \right) e^{\mu^\vee}\\
\stackrel{ \textrm{Eq } \eqref{eq:satake_def} }{=}
        & \sum_{\mu^\vee \in X^\vee} \delta^{\half}\left( \varpi^{\mu^\vee} \right) e^{\mu^\vee} \int_N dn \mathds{1}_{\left\{K \varpi^{\lambda^\vee} K\right\}} \left( \varpi^{\mu^\vee} n \right)\\
= \quad & \sum_{\mu^\vee \in X^\vee} \delta^{\half}\left( \varpi^{\mu^\vee} \right) e^{\mu^\vee} \int_K dk \int_N dn \mathds{1}_{\left\{K \varpi^{\lambda^\vee} K\right\}} \left( k \varpi^{\mu^\vee} n \right)\\
= \quad & \sum_{\mu^\vee \in X^\vee} \delta^{-\half}\left( \varpi^{\mu^\vee} \right) e^{\mu^\vee} \delta^{-1}\left( \varpi^{-\mu^\vee} \right) \int_K dk \int_N dn \mathds{1}_{\left\{K \varpi^{-\lambda^\vee} K\right\}} \left( n \varpi^{-\mu^\vee} k \right)\\
\stackrel{ \textrm{Iwasawa} }{=}
        & \sum_{\mu^\vee \in X^\vee} q^{-\langle \rho, \mu^\vee \rangle} e^{\mu^\vee} \int_{N \varpi^{-\mu^\vee} K \cap K \varpi^{-\lambda^\vee} K} dg \\
= \quad & \sum_{\mu^\vee \in X^\vee} q^{-\langle \rho, \mu^\vee \rangle} \Card_{G/K}\left( N \varpi^{-\mu^\vee} K \cap K \varpi^{-\lambda^\vee} K \right) e^{\mu^\vee} \ .
\end{align*}

The second relation can be found as equation (3.13) in \cite{Gross}. A formal argument is sketched as follows. We already know thanks to the Satake isomorphism that $\Sc\left( \lambda^\vee \right)$ is $W$-invariant. As a consequence, by expressing the spherical function in the basis of $G^\vee$-characters, there are coefficients $a_{\lambda^\vee, \mu^\vee}$ such that:
$$ \Sc\left( \lambda^\vee \right) = \sum_{\mu^\vee \in \left( X^\vee \right)^+} a_{\lambda^\vee, \mu^\vee} \ch V\left( \mu^\vee \right) \ . $$
Let us write $\mu^\vee \leq \lambda^\vee$ for $\lambda^\vee - \mu^\vee \in \left( Q^\vee \right)^+$. Because $N \varpi^{-\mu^\vee} K \cap K \varpi^{-\lambda^\vee} K = \emptyset$ unless $\mu^\vee \leq \lambda^\vee$ and $N \varpi^{-w_0 \lambda^\vee} K \cap K \varpi^{-\lambda^\vee} K$ is only one coset, we have in fact:
$$ \Sc\left( \lambda^\vee \right) = q^{\langle \rho, \lambda^\vee \rangle} \ch V\left( \lambda^\vee \right) + \sum_{\mu^\vee < \lambda^\vee} a_{\lambda^\vee, \mu^\vee} \ch V\left( \mu^\vee \right) \ .$$
If $\lambda^\vee$ is minuscule, only the first term remains.
\end{proof}
In the case of $\lambda^\vee$ not minuscule, $\Sc\left( \lambda^\vee\right)$ becomes a linear combination of characters $\ch V\left( \mu^\vee \right)$ for $\mu^\vee \leq \Lambda^\vee$. Such a change of basis involves Kazhdan-Lusztig polynomials and goes beyond the scope of the present work. Before diving into the probabilistic part of this section, we need to state the following.

\begin{corollary}
\label{corollary:support}
If $\lambda^\vee$ is dominant and $\mu^\vee \in W \lambda^\vee$, then:
$$ N \varpi^{-\mu^\vee} K \cap K \varpi^{-\lambda^\vee} K = N\left( \Oc \right) \varpi^{-\mu^\vee} K \ .$$
\end{corollary}
\begin{proof}
The fact that the right-hand side is included in the left-hand side is immediate. In order to prove the reverse inclusion, it suffices that both sets have the same cardinalities as $G/K$ cosets. On the one hand, thanks to the $W$-invariance in equation \eqref{eq:spherical_function} and the fact that $N \varpi^{-w_0 \lambda^\vee} K \cap K \varpi^{-\lambda^\vee} K $ is only one coset %(See (4) in theorem 2.6.11 in \cite{macdonald})
we have that:
\begin{align*}
  & \Card_{G/K}\left( N \varpi^{-\mu^\vee} K \cap K \varpi^{-\lambda^\vee} K \right)\\
= & q^{-\langle w_0\lambda^\vee - \mu^\vee, \rho \rangle} \Card_{G/K}\left( N \varpi^{-w_0 \lambda^\vee} K \cap K \varpi^{-\lambda^\vee} K \right)\\
= & q^{\langle \lambda^\vee - w_0 \mu^\vee, \rho \rangle} \ .
\end{align*}
On the other hand, if $n \in N\left( \Oc \right)$ then (\cite{Steinberg} lemmas 17 and 18), for any fixed order on the positive roots, there are unique parameters $\left( o_\beta \right)_{\beta \in \Phi^+}$ in $\Oc$ such that
$$ n = \prod_{\beta \in \Phi^+} x_{-\beta}\left( o_\beta \right).$$
By choosing an order such that:
$$ n = \prod_{ \substack{\beta \in \Phi^+ \\ \langle \beta, \mu^\vee \rangle \geq 0} } x_{-\beta}\left( o_\beta \right)
       \prod_{ \substack{\beta \in \Phi^+ \\ \langle \beta, \mu^\vee \rangle < 0   } } x_{-\beta}\left( o_\beta \right), $$
we obtain:
\begin{align*}
n \varpi^{-\mu^\vee} K & = \varpi^{-\mu^\vee} \varpi^{\mu^\vee} n \varpi^{-\mu^\vee} K \\
& = \varpi^{-\mu^\vee}
    \prod_{ \substack{\beta \in \Phi^+ \\ \langle \beta, \mu^\vee \rangle \geq 0} } x_{-\beta}\left( \varpi^{-\langle \beta, \mu^\vee \rangle} o_\beta \right)
    \prod_{ \substack{\beta \in \Phi^+ \\ \langle \beta, \mu^\vee \rangle < 0   } } x_{-\beta}\left( \varpi^{-\langle \beta, \mu^\vee \rangle} o_\beta \right) 
    K \\
& = \varpi^{-\mu^\vee}
    \prod_{ \substack{\beta \in \Phi^+ \\ \langle \beta, \mu^\vee \rangle \geq 0} } x_{-\beta}\left( \varpi^{-\langle \beta, \mu^\vee \rangle} o_\beta \right)
    K \ .
\end{align*}
Here, we write $\mu^\vee = w \lambda^\vee$. As the parameters $\left( \varpi^{-\langle \beta, \mu^\vee \rangle} o_\beta \mod \Oc \right)_{\beta \in \Phi^+}$ uniquely determine the coset $n \varpi^{-\mu^\vee} K$, we see that 
$$\Card_{G/K}\left( N \varpi^{-\mu^\vee} K \cap K \varpi^{-\lambda^\vee} K \right) = q^a , $$
with $a$ being the integer
$$
a = \sum_{ \substack{\beta \in \Phi^+ \\ \langle \beta, \mu^\vee \rangle \geq 0} } \langle \beta, \mu^\vee \rangle
  = \sum_{ \substack{\beta \in w^{-1} \Phi^+ \\ \langle \beta, \lambda^\vee \rangle \geq 0} } \langle \beta, \lambda^\vee \rangle
  = \sum_{ \substack{\beta \in w^{-1} w_0 \Phi^- \\ \langle \beta, \lambda^\vee \rangle \geq 0} } \langle \beta, \lambda^\vee \rangle \ .
$$
Since $\langle \beta, \lambda^\vee \rangle$ is either zero or negative when $\beta$ is a negative root, we can restrict the previous summation index to $\beta \in \Phi^+$ and discard the condition $\langle \beta, \lambda^\vee \rangle \geq 0$. Hence a summation over the inversion set $\Inv(w^{-1} w_0) := \Phi^+ \cap w^{-1} w_0 \Phi^{-}$. Using the fact that (consequence of corollary 1.3.22 in \cite{Kumar02})
$$ \forall v \in W, \sum_{\beta \in \Inv(v)} \beta = \rho - v \rho \ ,$$
we obtain:
$$ 
a = \left\langle \sum_{\beta \in \Phi^+ \cap w^{-1} w_0\Phi^-} \beta, \lambda^\vee \right\rangle 
  = \langle \rho - w^{-1} w_0 \rho, \lambda^\vee \rangle
  = \langle \rho, \lambda^\vee - w_0 \mu^\vee \rangle \ .
$$
\end{proof}

\subsection{A spherical and a solvable random walk}
The elements of the spherical Hecke algebra can be seen as bi-invariant probability measures. This allows the construction of random walks on the group, for which the Whittaker functions should be harmonic. More precisely, we will work on a slightly different level, on the solvable group $B$. Because of the Iwasawa ($NAK$) decomposition, one identifies $G/K$ to $B\left( \Kc \right) / B\left( \Oc \right)$. Therefore harmonic properties for functions in $\Fc\left( G/K \right)$ have counterparts for functions in $\Fc\left( B \right)$. For every discrete time Markov process $X$, we will denote by $\Fc^X = \left( \Fc^X_t \right)_{t \in \N}$ the natural filtration generated by $X$.

Let $\lambda^\vee$ be a dominant cocharacter. Define $\left( G_t; t \geq 0 \right)$ as the left invariant random walk on $G$ with \iid~increments distributed as $\Uc\left( K \right) \varpi^{-\lambda^\vee} \Uc\left( K \right)$:
$$ G_t := g_1 g_2 \dots g_t \quad \textrm{ and } \quad \forall s \in \N, \ g_s \eqlaw \Uc\left( K \right) \varpi^{-\lambda^\vee} \Uc\left( K \right) \ .$$
The previous product has to be understood as a product of independent random variables. As a general convention, the appearance of several $\Uc\left( H \right)$ in the same expression will always refer to \emph{different} independent Haar distributed random variables on $H$. Also, we will call $\left( G_t; t \geq 0 \right)$ the spherical random walk since its increments are $K$ bi-invariant. The following proposition summarizes its properties.

\begin{proposition}
\label{proposition:G_harmonicity}
The law of the random variable $\Uc\left( K \right) \varpi^{-\lambda^\vee} \Uc\left( K \right)$ is given as follows. For every $f \in \Fc(G)$:
$$ \E\left[ f\left( \Uc\left( K \right) \varpi^{-\lambda^\vee} \Uc\left( K \right) \right) \right]
 = \frac{1}{\Card_{G/K}\left( K \varpi^{-\lambda^\vee} K \right)} \int_G f\left( g \right) \mathds{1}_{ \{ g \in K \varpi^{-\lambda^\vee} K \}} dg \ ,$$
and the fact that the Whittaker function is a convolution eigenfunction can be restated as an $\alpha$-harmonicity property:
$$ \E\left( \Wc_\chi\left( G_{t+1} \right) \ | \ \Fc^G_t, G_t = g \right)
 = \frac{ \Sc\left( \lambda^\vee \right)(\chi) }
        { \Card_{G/K}\left( K \varpi^{-\lambda^\vee} K \right) } \Wc_\chi(g) \ .
$$
\end{proposition}
\begin{proof}
Let us give two proofs of the first statement. Clearly the law of our random variable is absolutely continuous with respect to the Haar measure on $G$: because $K$ is a compact open subgroup, the Haar measure on $K$ is nothing but $\mathds{1}_{\{g \in K\}} dg$. Thus, the expectation $\E\left[ f\left( \Uc\left( K \right) \varpi^{-\lambda^\vee} \Uc\left( K \right) \right) \right]$ can be written as an integral against $dg$. The Radon-Nikodym derivative has to be a $K$ bi-invariant function, and hence proportional to $\mathds{1}_{ \{ g \in K \varpi^{-\lambda^\vee} K \}}$. The proportionality constant is given by the volume of $K \varpi^{-\lambda^\vee} K$ or equivalently the number of right $K$ cosets it contains. Hence a first proof.

The second proof consists in checking the formula by doing the computation backwards. Using the left and right invariance of the Haar measure, we obtain:
\begin{align*}
  \quad & \int_G f\left( g \right) \mathds{1}_{ \{ g \in K \varpi^{-\lambda^\vee} K \}} dg\\
= \quad & \int_K dk_1 \ \int_K dk_2 \int_G f\left( k_1 g k_2\right) \mathds{1}_{ \{ g \in K \varpi^{-\lambda^\vee} K \}} dg \\
\stackrel{\textrm{Fubini}}{=} 
        & \int_G dg \ \mathds{1}_{ \{ g \in K \varpi^{-\lambda^\vee} K \}} \int_K dk_1 \ \int_K dk_2  f\left( k_1 g k_2\right)\\
= \quad & \int_G dg \ \mathds{1}_{ \{ g \in K \varpi^{-\lambda^\vee} K \}} \int_K dk_1 \ \int_K dk_2  f\left( k_1 \varpi^{-\lambda^\vee} k_2\right)\\
= \quad & \Card_{G/K}\left( K \varpi^{-\lambda^\vee} K \right) \int_K dk_1 \ \int_K dk_2  f\left( k_1 \varpi^{-\lambda^\vee} k_2\right)\\
= \quad & \Card_{G/K}\left( K \varpi^{-\lambda^\vee} K \right) \E\left( f\left( \Uc\left( K \right) \varpi^{-\lambda^\vee} \Uc\left( K \right) \right) \right) \ .
\end{align*}

For the second statement, $\Fc^G$ being the natural filtration of the random walk $\left( G_t; t \in \N \right)$. We have:
\begin{align*}
  \quad & \E\left[ \Wc_\chi\left( G_{t+1} \right) \ | \ \Fc^G_t, G_t = g \right]\\
= \quad & \E\left[ \Wc_\chi\left( g g_{t+1} \right) \right]\\
= \quad & \frac{1}{\Card_{G/K}\left( K \varpi^{-\lambda^\vee} K \right)} \int_G \Wc_\chi(g h ) \mathds{1}_{ \{ h \in K \varpi^{-\lambda^\vee} K \}} dh\\
\stackrel{ G \textrm{ unimodular} }{=}
        & \frac{1}{\Card_{G/K}\left( K \varpi^{-\lambda^\vee} K \right)} \int_G \Wc_\chi(g h^{-1} ) \mathds{1}_{ \{ h \in K \varpi^{\lambda^\vee} K \}} dh\\
= \quad & \frac{1}{\Card_{G/K}\left( K \varpi^{-\lambda^\vee} K \right)} \left( \Wc_\chi *_G  \mathds{1}_{ \{ K \varpi^{\lambda^\vee} K \}} \right)(g)\\
\stackrel{ \textrm{Eq } \eqref{eq:convolution_eigeneq} }{=}
        & \frac{ \Sc\left( \lambda^\vee \right)\left( \chi \right) }
               { \Card_{G/K}\left( K \varpi^{-\lambda^\vee} K \right) } \Wc_\chi(g) \ .
\end{align*}
\end{proof}

Now, let us consider a random walk $\left( B_t; t \in \N \right)$ on the solvable group $B$ which will inherit the $\alpha$-harmonicity property of $\left( G_t; t \in \N \right)$ given in Proposition \ref{proposition:G_harmonicity}. When considering the Iwasawa decomposition of the random walk $\left( G_t; t \in \N \right)$, we write:
$$ G_t := B_t K_t ,$$
but since $B_t$ is defined only modulo $B\left( \Oc \right)$, a choice has to be made. Nevertheless, this choice happens only in a realization of our random walk. In term of distribution, it is natural to consider increments for $B_t$ that are right $B\left( \Oc \right)$ invariant. Therefore we choose $\left( B_t; t \in \N \right)$ to be a left-invariant random with independent increments
$$ B_t := b_1 b_2 \dots b_t$$
with the distribution of each $b_s$ being right $B\left( \Oc \right)$ invariant. Hence we are left with the task of specifying only the distribution of $b_s \cdot B\left( \Oc \right)$ on $B / B\left( \Oc \right) \approx G/K$. This is given by $b_s K \eqlaw g_s K$. Therefore, the increments $\left( b_t \ ; t \in \N \right)$ are $B\left( \Oc \right)$ bi-invariant. 

We easily see that $\Wc_\chi$ inherits the $\alpha$-harmonicity property with respect to $\left( B_t; t \in \N \right)$ from the corresponding property with respect to $\left( G_t; t \in \N \right)$:
\begin{align}
\label{eq:B_harmonicity}
\E\left( \Wc_\chi\left( B_{t+1} \right) \ | \ \Fc^B_t, B_t = b \right)
   = \frac{ \Sc\left( \lambda^\vee \right)(\chi) }
        { \Card_{G/K}\left( K \varpi^{-\lambda^\vee} K \right) } \Wc_\chi\left( b \right) \ .
\end{align}
Indeed, by the tower rule of conditional expectation:
\begin{align*}
\E\left( \Wc_\chi\left( B_{t+1} \right) \ | \ \Fc^B_t, B_t = b \right) & = \E\left( \Wc_\chi\left( G_{t+1} \right) \ | \ \Fc^B_t, B_t = b \right)\\
& = \E\left( \E\left( \Wc_\chi\left( G_{t+1} \right) \ | \ \Fc^G_t \right) \ | \ \Fc^B_t, B_t = b \right)\\
& = \frac{ \Sc\left( \lambda^\vee \right)(\chi) }
        { \Card_{G/K}\left( K \varpi^{-\lambda^\vee} K \right) } \Wc_\chi\left( b \right) \ .
\end{align*}

The law of the increments $\left( B_t; t \in \N \right)$ is only explicit in the case of a spherical random walk with a choice of dominant cocharacter $\lambda^\vee$ that is minuscule. From now on, we will use the letter $\Lambda^\vee$ to refer to the minuscule cocharacter we fixed at the beginning.
%But is is clearly left and right $N\left( \Oc \right)$ invariant. Moreover, in the $NA$ decomposition, the $A$ part is $W$ invariant. Write $b_s = \varpi^{\mu^\vee} n_s$. The law of $\mu^\vee$ must be $W$-invariant. As the Satake transform counts the number of single cosets in $K \varpi^{\Lambda^\vee} K$ such that the diagonal part has fixed valuation, and from \cite{Cartier}, p. 148, we see that if $\mu^\vee$ is dominant then $\mu^\vee \leq \Lambda^\vee$. There should be an argument using representation theory.
\begin{proposition}
\label{proposition:B_increments}
If the increments of $\left( G_t; t \in \N \right)$, $\left( g_1, g_2, \dots \right)$ are distributed as $\Uc\left( K \right) \varpi^{-\Lambda^\vee} \Uc\left( K \right)$, with $\Lambda^\vee$ minuscule, then the increments of $\left( B_t; t \in \N \right)$ are also independent and identically distributed with:
\begin{align}
\label{eq:B_increment_law}
B_t^{-1} B_{t+1} \eqlaw \Uc\left( B\left( \Oc \right) \right) \varpi^{-\mu^\vee_\rho} \Uc\left( B\left( \Oc \right) \right)
\end{align}
with the above product being a product of three independent random variables and $\mu^\vee_\rho$ being distributed according to
\begin{align}
\label{eq:mu_rho}
    \forall \mu^\vee \in W \Lambda^\vee, \P\left( \mu^\vee_\rho = \mu^\vee \right) 
& = \frac{ q^{\langle \Lambda^\vee - w_0 \mu^\vee, \rho \rangle} }
         { \Card_{G/K}\left( K \varpi^{-\Lambda^\vee} K \right) } \ .
\end{align}
\end{proposition}
\begin{proof}
Let $b \in B$ be a random increment of $\left( B_t; t \in \N \right)$. Because $\Lambda^\vee$ is minuscule, $N \varpi^{-\mu^\vee} K \cap K \varpi^{-\Lambda^\vee} K = \emptyset$ unless $\mu^\vee \in W \Lambda^\vee$. Hence, there exists a random $\mu^\vee_\rho \in W \Lambda^\vee$ so that $b \in N \varpi^{-\mu^\vee_\rho} K$ in the Iwasawa decomposition. From corollary \ref{corollary:support}, we even have $b \in N\left( \Oc \right) \varpi^{-\mu^\vee_\rho} K$. At this moment, the support of the distribution of $b$ is well identified.

Now, write $b = b_1 \varpi^{-\mu^\vee_\rho} b_2$ with $b_1, b_2 \in B\left( \Oc \right)$ - not a unique expression. As the distribution of $b$ is bi-invariant under $B\left( \Oc \right)$, we see that $b_1$ and $b_2$ can be picked Haar distributed on $B\left( \Oc \right)$. This gives equation \eqref{eq:B_increment_law}. 

We are only left with the task of identifying the law of $\mu_\rho^\vee$. As $bK = g K$ with $g \eqlaw \Uc\left( K \right) \varpi^{-\Lambda^\vee} \Uc\left( K \right)$, $bK$ is a uniformly chosen right $K$ coset in $K \varpi^{-\Lambda^\vee} K$. Therefore, for a fixed $\mu^\vee$, $\P\left( \mu^\vee_\rho = \mu^\vee \right)$ is proportional to the number of such cosets that are of the form $N \varpi^{-\mu^\vee} K$. This number, $q^{\langle \Lambda^\vee - w_0 \mu^\vee, \rho \rangle} $, is given in the proof of corollary \ref{corollary:support}, hence the formula in equation \eqref{eq:mu_rho}.
\end{proof}

\subsection{Penalization and proof of Theorem \ref{thm:main_result2}}
\label{subsection:proof_thm_2}

We will now penalize the random walk $\left( B_t; t \in \N \right)$ by its diagonal part in order to obtain $\left( B_t\left( W^{(z)} \right) ; t \in \N \right)$. The following lemma tells us how the distribution of their increments are related.

\begin{lemma}
\label{lemma:penalisation}
For all $f \in \Fc\left( B \right)$, we have that:
\begin{align*}
  & \E\left[ f\left( B_t^{-1} B_{t+1} \right) \right]\\
= & \frac{ q^{\langle \Lambda^\vee, \rho \rangle} \ch V\left( \Lambda^\vee \right)(z) }
         { \Card_{G/K}\left( K \varpi^{-\Lambda^\vee} K \right) } 
    \E\left[ \left( \chi^{-1} \delta^{-\half} f \right)\left( B_t^{-1}\left( W^{(z)} \right) B_{t+1}\left( W^{(z)} \right) \right) \right] \ .
\end{align*}
\end{lemma}
\begin{proof}
We have:
\begin{align*}
  & \E\left[ f\left( B_t^{-1} B_{t+1} \right) \right]\\
\stackrel{ \textrm{Prop. } \ref{proposition:B_increments}}{=} \ 
  & \sum_{\mu^\vee \in W \Lambda^\vee} 
    \frac{ q^{\langle \Lambda^\vee - w_0 \mu^\vee, \rho \rangle} }
         { \Card_{G/K}\left( K \varpi^{-\Lambda^\vee} K \right) } 
    \E\left[ f\left( \Uc\left( B\left( \Oc \right) \right) \varpi^{-\mu^\vee} \Uc\left( B\left( \Oc \right) \right) \right) \right]\\
\stackrel{ \textrm{Eq. } (\ref{eq:def_chi_z}, \ref{eq:def_modular}) }{=}
  & \sum_{\mu^\vee \in W \Lambda^\vee} 
    \frac{ q^{\langle \Lambda^\vee, \rho \rangle} e^{\langle z, \mu^\vee \rangle}
    \E\left[ \left( \chi^{-1} \delta^{-\half} f \right)\left( \Uc\left( B\left( \Oc \right) \right) \varpi^{-\mu^\vee} \Uc\left( B\left( \Oc \right) \right) \right) \right] }
         { \Card_{G/K}\left( K \varpi^{-\Lambda^\vee} K \right) }\\
\stackrel{ \textrm{Eq. } \eqref{eq:rw_lattice} }{=} \ 
  & \frac{ q^{\langle \Lambda^\vee, \rho \rangle} \ch V\left( \Lambda^\vee \right)(z) }
         { \Card_{G/K}\left( K \varpi^{-\Lambda^\vee} K \right) } 
    \E\left[ \left( \chi^{-1} \delta^{-\half} f \right)\left( B_t^{-1}\left( W^{(z)} \right) B_{t+1}\left( W^{(z)} \right) \right) \right] \ .
\end{align*} 
\end{proof}

Now, the $\alpha$-harmonicity in equation \eqref{eq:B_harmonicity} will translate to a strict harmonicity property, which was the basis of our Poisson kernel formula. For more convenient notations, let us define:
$$ \forall b \in B, f_z\left( b \right) := \delta\left( b \right)^{-\half} \chi\left( b \right)^{-1} \Wc_z\left( b \right) \ .$$
Then we have:
\begin{align*}
  \quad \ & \E\left( f_z\left( B_{t+1}\left( W^{(z)} \right) \right) \ | \ \Fc^B_t, B_t = b \right)\\
= \quad \ & \E\left( f_z\left( b B_t^{-1}\left( W^{(z)} \right) B_{t+1}\left( W^{(z)} \right) \right) \right)\\
= \quad \ & \E\left[ \left( \chi^{-1} \delta^{-\half} \Wc_z\right)\left[ b B_t^{-1}\left( W^{(z)} \right) B_{t+1}\left( W^{(z)} \right) \right] \right]\\
\stackrel{ \textrm{Lemma } \ref{lemma:penalisation}}{=}
          & \frac{ \Card_{G/K}\left( K \varpi^{-\Lambda^\vee} K \right) } 
               { q^{\langle \Lambda^\vee, \rho \rangle} \ch V\left( \Lambda^\vee \right)(z) }
          \delta\left( b \right)^{-\half} \chi\left( b \right)^{-1}
          \E\left[ \Wc_z\left( b B_t^{-1} B_{t+1} \right) \right]\\
\stackrel{ \textrm{Eq } \eqref{eq:B_harmonicity}}{=}
        \ &  \frac{ \Card_{G/K}\left( K \varpi^{-\Lambda^\vee} K \right) } 
                  { q^{\langle \Lambda^\vee, \rho \rangle} \ch V\left( \Lambda^\vee \right)(z) }
             \delta\left( b \right)^{-\half} \chi\left( b \right)^{-1}
	  \frac{ \Sc\left( \Lambda^\vee \right)(\chi) }
	      { \Card_{G/K}\left( K \varpi^{-\Lambda^\vee} K \right) }
    \Wc_z\left( b \right)\\
= \quad & \frac{ \Sc\left( \Lambda^\vee \right)(\chi) } 
               { q^{\langle \Lambda^\vee, \rho \rangle} \ch V\left( \Lambda^\vee \right)(z) }
    f_z\left( b \right)\\
\stackrel{ \textrm{Eq } \eqref{eq:spherical_function_minuscule}}{=}
  & f_z\left( b \right) \ .
\end{align*}
This concludes the proof of Theorem \ref{thm:main_result2}.

% --------------------------------------------------------------------
\section{Existence of \texorpdfstring{$N_\infty\left( W^{(z)} \right)$}{limit in N} }
\label{section:N_infty_CV}

In this section, we obtain the existence of the limiting random variable $N_\infty\left( W^{(z)} \right)$ (Proposition \ref{proposition:N_infty_CV}) as a consequence of the deterministic Theorem \ref{thm:CV}. The latter states that a sequence in $N$ with multiplicative increments converging to the identity must converge. It is an analogue of the standard fact that series in ultrametric spaces converge as soon as their general term goes to zero. Such a result is intuitive if one represents $N_t\left( W^{(z)} \right)$ in a group of upper triangular matrices thanks to the Lie-Kolchin theorem. Indeed the coefficients on top of the diagonal will follow an additive series which converges without difficulty because of the ultrametric topology. Using a recurrence, one hopes to propagate the result to other coefficients in stages. Morally, this is the strategy of proof we follow, even if we choose to work intrinsically.

Let $\height: \Phi^+ \rightarrow \N^*$ be the height function, meaning that if $\beta = \sum_{\alpha \in \Delta} c_\alpha \alpha\in \Phi^+$ then the height of $\beta$ is $\height(\beta) = \sum_{\alpha \in \Delta} c_\alpha$.  Now, recall that the group $N$ is generated by one parameter subgroups $\left( x_{-\beta}\left( \Kc \right) ; \beta \in \Phi^+ \right)$. For every $i \in \N$, define the subgroups:
$$ N^i := \langle x_{-\beta}\left( \Kc \right) ; \height(\beta) \geq i \rangle \ .$$

\begin{theorem}
\label{thm:CV} 
Let $i \in \N$ and $\left( N^{i}_t ; t=0,1,2, \dots \right)$ be any sequence in $N^i$. Set:
\begin{align}
\label{eq:incr_ni}
\forall t \in \N, \ N^{i}_{t+1} & = N^{i}_{t} \Delta^i_t \ .
\end{align}
If
$$ \Delta^i_t \stackrel{t \rightarrow \infty}{\longrightarrow} \Id \ ,$$
then $\left( N^{i}_t ; t=0,1,2, \dots \right)$ is convergent.
\end{theorem}

The proof will require knowledge of the structure of the groups $N^i$. This is obtained from Steinberg's presentation in terms of generators and relations.

\begin{theorem}[ \cite{Steinberg} p. 22-26 ]
\label{thm:structure_N}
The following assertions hold:
\begin{enumerate}
 \item[(i)] For every $i \in \N$, the subgroup $N^i$ is a normal subgroup of $N$. In particular $N^{i+1}$ is normal in $N^{i}$.
 \item[(ii)] $\left[N, N^i \right] \subset N^{i+1}$. In particular $\left[N^i, N^i \right] \subset N^{i+1}$ and $N^{i+1} \backslash N^i$ is abelian.
 \item[(iii)] Fix an order on the roots of height $i$. Then for every $n^i \in N^i$, there exists a unique $n^{i+1} \in N^{i+1}$ and unique coefficients $\left( \theta^\beta \right)_{\left( \height(\beta)=i \right)}$ in the field $\Kc$ such that
       \begin{align}
       \label{eq:n_i_factorisation}
       n^i = & n^{i+1} n^{\widehat{i}}
       \end{align}
       where 
       $$ n^{\widehat{i}} = \prod_{\height(\beta)=i} x_{-\beta}\left( \theta^\beta \right). $$
       Th product respects the previously fixed order. Moreover, $n^{i+1} \in N^{i+1}$ may depend on the choice of order on the roots $\left( \beta ; \height(\beta)=i \right)$ but the coefficients $\left( \theta^\beta \right)_{\left( \height(\beta)=i \right)}$ do not.
 \item[(iv)] The map
       $$ \begin{array}{cccc}
               \varphi^i: & N^i & \rightarrow & \Kc^{ \Card\left\{ \beta, \ \height(\beta) = i \right\} }\\
                          & n^i & \mapsto     & \left( \theta^\beta \right)_{\left( \height(\beta)=i \right)}
               \end{array}
       $$
       is a continuous group homomorphism.
\end{enumerate}
\end{theorem}
\begin{proof}
Let us give pointers to the proof. The two first claims are entirely contained in \cite{Steinberg} Corollary 4 p.26. For the third claim, existence and uniqueness of the factorisation \eqref{eq:n_i_factorisation} is Lemma 17 p.24 from \cite{Steinberg}. To see that the coefficients $\left( \theta^\beta \right)_{\left( \height(\beta)=i \right)}$ do not depend on the order, let us consider the product $\prod_{\height(\beta)=i} x_{-\beta}\left( \theta^\beta \right) \in N^i$ for two different orders. Since $N^{i+1} \backslash N^i$ is abelian, these two products are equal modulo $N^{i+1}$ on the left, hence the result.

For the last claim, the map $\varphi^i$ is defined unambiguously thanks to (iii). In order to prove the homomorphism property:
$$ \forall \left( x,y \right) \in N^i \times N^i, \ \varphi^i\left( xy \right) =  \varphi^i\left( x\right) + \varphi^i\left( y \right),$$
we can reduce the problem to $x = x_{-\beta}(t)$ for a certain $\beta$ with $\height(\beta)=i$. Indeed, one can factor $x = x^{i+1} x^{\widehat{i}}$ as in equation \eqref{eq:n_i_factorisation}, discard the term in $x^{i+1} \in N^{i+1}$ because $\varphi^i$ is left $N^{i+1}$ invariant and treat each term in the product $x^{\widehat{i}}$ separately. We can reduce further to $y = y^{\widehat{i}} = \prod_{\height(\gamma)=i} x_{-\gamma}\left( \theta^\gamma \right)$ with an order such that the root in the left-most term is $\gamma = \beta$.  Indeed, write $y = y^{i+1} y^{\widehat{i}}$ and notice that:
$$ \varphi^i\left( x y \right) = \varphi^i\left( x y^{i+1} x^{-1} x y^{\hat{i}} \right)
                               = \varphi^i\left( x y^{\widehat{i}} \right),
$$
as $x y^{i+1} x^{-1} \in N^{i+1}$. After these two reductions, the statement becomes immediate.
\end{proof}

\begin{proof}[ {\bf Proof of Theorem \ref{thm:CV}} ]
We proceed by backward induction on the index $i$.

Because $\Phi$ is finite, there exists a maximal height and $N^i = \{\Id\}$ for $i \gg 1$. Therefore, for $i$ large enough, any sequence $N^i_t$ is stationary and convergence is immediate. 

Now let us fix an index $i$ and assume by induction hypothesis that the theorem holds for the index $i+1$. Let $\left( N^{i}_t ; t=0,1,2, \dots \right)$ be a sequence as in Theorem \ref{thm:CV}. Thanks to the second claim in Theorem \ref{thm:structure_N}, there are two sequences $N_t^{i+1} \in N^{i+1}$ and $N_t^{\widehat{i}} \in N^{i}$ such that:
\begin{align}
\label{eq:sequence_factorisation}
N_t^{i}      =  & N_t^{i+1} N_t^{\widehat{i}} \ ,
\end{align}
where
$$ N_t^{\widehat{i}} = \prod_{\height(\beta)=i} x_{-\beta}\left( \theta^\beta_t \right) \ .$$
The proof will be finish once we establish that both sequences $N_t^{i+1}$ and $N_t^{\widehat{i}}$ converge.

On the one hand, $\left( \theta^\beta_t \right)_{\left( \height(\beta)=i \right)} = \varphi^i\left( N^i_{t} \right) = \varphi^i\left( N^{\widehat{i}}_{t} \right)$ where $\varphi^i$ is the group homomorphism defined in Theorem \ref{thm:structure_N}. Clearly, convergence for $N^{\widehat{i}}_{t}$ is equivalent to the convergence of $\varphi^i\left( N^{\widehat{i}}_{t} \right)$. By applying $\varphi^i$ to equation \eqref{eq:incr_ni}, we obtain:
$$ \forall t \in \N, \ \varphi^i\left( N^i_{t+1} \right) = \varphi^i\left( N^i_{t} \right) + \varphi^i\left( \Delta^i_{t} \right) \ .$$
Recall $\Delta^i_{t} \stackrel{t \rightarrow \infty}{\longrightarrow} \Id$ and therefore $\varphi^i\left( \Delta^i_{t} \right) \stackrel{t \rightarrow \infty}{\longrightarrow} 0$. Because of the ultrametric topology on $\Kc^{ \Card\left\{ \beta, \ \height(\beta) = i \right\} }$, the sequence $\varphi^i\left( N^i_{t} \right)$ converges. Notice that this part of the proof is virtually the same as in the rank $1$ example \ref{example:PGL2}.

On the other hand, combining recurrence \eqref{eq:incr_ni} and equation \eqref{eq:sequence_factorisation} yields
$$ N_{t+1}^{i+1} = N_{t}^{i+1} N_t^{\widehat{i}} \Delta^{i}_t \left( N_{t+1}^{\widehat{i}} \right)^{-1} \ .$$
By setting $\Delta^{i+1}_t := N_t^{\widehat{i}} \Delta^{i}_t \left( N_{t+1}^{\widehat{i}} \right)^{-1} \stackrel{t \rightarrow \infty}{\longrightarrow} \Id$ , we see that the induction hypothesis applies to the sequence $N_t^{i+1}$, which is therefore convergent.
\end{proof}

% --------------------------------------------------------------------
%\section{Concluding remarks}
%\label{section:concluding_remarks}

% --------------------------------------------------------------------
\section{Acknowledgement}
The author is grateful to Dan Bump, Paul-Olivier Dehaye and Simon Pepin-Lehalleur for guidance and many helpful discussions. Also, I would like to mention Gautham Chinta and Philippe Bougerol for their encouragement in pursuing this probabilistic approach. 

% --------------------------------------------------------------------
\small
\bibliographystyle{halpha}
\bibliography{Bib_CS}

\end{document}